\documentclass{amsart}
\usepackage{amscd,amssymb}
\usepackage[mathcal]{euscript}
\usepackage{graphicx}
\usepackage{hyperref}
\usepackage[arrow, matrix, curve]{xy}
\numberwithin{equation}{section}

\newtheorem{Thm}[equation]{Theorem}
\newtheorem{Lemma}[equation]{Lemma}
\newtheorem{Prop}[equation]{Proposition}

\newtheorem{Cor}[equation]{Corollary}
\theoremstyle{definition}
\newtheorem{Def}[equation]{Definition}
\newtheorem{Rmk}[equation]{Remark}
\newtheorem{ex}[equation]{Example}

\newcommand{\N}{\mathbb{N}}
\newcommand{\Z}{\mathbb{Z}}
\newcommand{\id}{\mathrm{id}}
\newcommand{\map}{\mathrm{map}}

\newcommand{\sset}{{\mathcal{S}_\ast}}
\DeclareMathOperator*{\colim}{colim}
\DeclareMathOperator*{\nlim}{lim}

\DeclareMathOperator{\Map}{\mathrm{Map}}

\DeclareMathOperator{\sk}{\mathrm{sk}}
\DeclareMathOperator{\csk}{\mathrm{csk}}
\DeclareMathOperator{\Inj}{\mathrm{Inj}}
\DeclareMathOperator{\Surj}{\mathrm{Surj}}
\DeclareMathOperator{\Bij}{\mathrm{Bij}}
\newcommand{\Sp}{\mathbb{S}}
\DeclareMathOperator{\conn}{\mathrm{conn}}

\theoremstyle{remark}
\newtheorem*{ass}{Assumptions}

\begin{document}

\author{Dominik Ostermayr}
\address{Mathematisches Institut der Universit\"at zu K\"oln, Weyertal 86-90, 50931 K\"oln, Germany}
\email{dosterma@math.uni-koeln.de }
\title{Equivariant $\Gamma$-spaces}
\begin{abstract}
The aim of this note is to provide a comprehensive treatment of the homotopy theory of $\Gamma$-$G$-spaces
for $G$ a finite group.
We introduce two level and stable model structures on $\Gamma$-$G$-spaces
and exhibit Quillen adjunctions to $G$-symmetric spectra
with respect to a flat level and a stable flat model structure respectively.
Then we give a proof that $\Gamma$-$G$-spaces model connective equivariant stable homotopy theory
along the lines of the proof in the non-equivariant setting given by Bousfield and Friedlander~\cite{BF}.
Furthermore, we study the smash product of $\Gamma$-$G$-spaces and show that
the functor from $\Gamma$-$G$-spaces to $G$-symmetric spectra commutes with the derived smash product.
Finally, we show that there is a good notion of geometric fixed points for $\Gamma$-$G$-spaces.
\end{abstract}
\maketitle
\setcounter{tocdepth}{1}
\tableofcontents
\thispagestyle{empty}
\newpage

\begin{section}{Introduction}
\label{sec: Introduction}
In the classical paper~\cite{Segal 1} Segal introduced $\Gamma$-spaces
as a tool to produce infinite loop spaces.
In fact, Segal showed that $\Gamma$-spaces model
connective stable homotopy theory and later Bousfield and Friedlander proved this in the
language of model categories~\cite{BF}. 

For $G$ a finite group, Segal developed the machinery of $\Gamma$-$G$-spaces
in~\cite{Segal 3} and it is known that \emph{very special} $\Gamma$-$G$-spaces
give rise to equivariant infinite loop spaces (cf.~\cite{Segal 3},~\cite{Shimakawa},~\cite{Santhanam}).
Santhanam also proved that $\Gamma$-$G$-spaces with a suitable model structure
are equivalent to equivariant $E_\infty$-spaces (cf.~\cite{Santhanam}).

Using the results of Shimakawa~\cite{Shimakawa}, we give a proof along the lines of~\cite{BF} that $\Gamma$-$G$-spaces
model connective equivariant stable homotopy theory.
Moreover, $\Gamma$-$G$-spaces possess a symmetric monoidal smash product as was shown by Lydakis~\cite{Lydakis},
motivating the question if this equivalence can be realized by a Quillen functor to a symmetric monoidal
category of $G$-spectra which commutes with the derived smash product.
This turns out to be true, if one uses the flat model structures on $G$-symmetric spectra as constructed by Hausmann~\cite{Hausmann}.
Even non-equivariantly, this might be of interest on its own right.
In addition, we define a geometric fixed point functor for $\Gamma$-$G$-spaces
which has all desirable properties.

The structure of the paper is as follows.
Sections $2$ and $3$ contain a brief review of basic facts about $G$-equivariant homotopy
theory and $G$-symmetric spectra. In particular, we will introduce the flat
model structure.
In Section $4$, we briefly discuss basic definitions and constructions concerning $\Gamma$-$G$-spaces and 
introduce two level model structures. The projective model structure was employed by Santhanam
in~\cite{Santhanam}, too, but we also show how to generalize the \emph{strict} model structure of~\cite{BF} to the equivariant setting.
In Section $5$, we exhibit Quillen pairs between the level model structures on $\Gamma$-$G$-spaces and the flat level model structure on
$G$-symmetric spectra. This requires a characterization of flat cofibrations of $G$-symmetric spectra which we carry out in the
appendix. We also show that spectra obtained from $\Gamma$-$G$-spaces are equivariantly connective and, using the results of~\cite{Shimakawa},
we show that very special $\Gamma$-$G$-spaces give rise to  $G\Omega$-symmetric spectra up to a level fibrant replacement.
After these preparations, we show that the homotopy categories with respect to the level model structures of very special $\Gamma$-$G$-spaces
and those connective spectra which are level equivalent to $G\Omega$-spectra are equivalent.
In Section $6$, we introduce stable equivalences of $\Gamma$-$G$-spaces and the stable model structures on $\Gamma$-$G$-spaces
corresponding to the two level model structures.
This leads to the equivalence of the homotopy categories
of $\Gamma$-$G$-spaces and connective $G$-symmetric spectra with respect to the stable model structures.
Section $7$ contains a discussion of the smash product of $\Gamma$-$G$-spaces.
Following the non-equivariant results from~\cite{Lydakis}, we show that it is well-behaved with respect to the model structures
and the functor from $\Gamma$-$G$-spaces to $G$-symmetric spetra commutes with the derived
smash product. 
Finally, in Section $8$, we define geometric fixed points for $\Gamma$-$G$-spaces with respect to a subgroup $H\leqslant G$.
This is a lax symmetric monoidal functor which sends suspension spectra to suspension spectra
and commutes with the derived smash product up to stable equivalence.
We characterize stable equivalences of $\Gamma$-$G$-spaces as those maps which induce stable equivalences
on all geometric fixed points.

\addtocontents{toc}{\protect\setcounter{tocdepth}{1}}
\subsection*{Acknowledgements}
This paper grew out of the author's Master's thesis written at the University of Bonn.
I want to thank my supervisor Stefan Schwede for drawing my interest towards the homotopy theory
of $\Gamma$-$G$-spaces and sharing his insights. In particular, the construction of geometric fixed points
is due to him.
Furthermore, I would like to thank Steffen Sagave for pointing out the characterization of flat $\mathcal{I}$-spaces
in~\cite{Sagave Schlichtkrull}. This led to the results in Section \ref{subsec: A characterization of flat cofibrations}.
I also would like to thank Markus Hausmann for sharing his Master's thesis and answering several questions.
Finally, I want to thank the GRK 1150 ``Homotopy and Cohomology'' for financial support during the preparation of this paper.
\addtocontents{tox}{\protect\setcounter{tocdepth}{2}}
\end{section}

\newpage
\begin{section}{Recollections on equivariant homotopy theory}
\label{sec: Recollections on equivariant homotopy theory}
This section is intended to fix basic terminology about model categories and recollect facts
about $G$-equivariant homotopy theory.
If not stated otherwise, the source of material is~\cite[II. 1., III. 1.]{MM}
though we work with simplicial sets as opposed to topological spaces.
Throughout this paper, the word ``space'' will mean simplicial set.

\begin{subsection}{Model categorical notions}
\label{subsec: Model categorical notions}
We will freely use the concepts of model category theory.
A reference is~\cite{DS}
and we use the numbering therein when referring to the axioms $\mathbf{MC 1}$ up to
$\mathbf{MC 5}$.
Other concepts from model category theory which we will make use of several times
include the small object argument and related notions like
$I$-injectives, $I$-cofibrations and regular $I$-cofibrations denoted $I$-inj,
$I$-cof and $I$-$\text{cof}_{\text{reg}}$
respectively. For a brief discussion of those we refer to
\cite[Appendix A]{sthomalg}.
A model category is \emph{proper} if weak equivalences are preserverd
under pullbacks along fibrations and under pushouts along cofibrations. A couple of times we
will make use of Reedy's patching lemma (cf.~\cite[3.8]{BF}).
We will assume basic knowledge about the homotopy theory of the category of spaces.
\end{subsection}

\begin{subsection}{Model structures on $G$-spaces}
\label{subsec: Model structures on G spaces}
The category of based spaces will be denoted
$\sset$.
It is closed symmetric monoidal under the smash product with unit $S^0$.
If $G$ is a finite group,
we also have the category of based spaces with left $G$-action together
with (not necessarily equivariant) maps
$\sset_G$. It is enriched over $G\sset$,
the corresponding category enriched over
$\sset$ with same objects but equivariant maps.

Several model structures on $G$-spaces will play a role in this paper.
The following model structure on $G\sset$ is the most important one.
A morphism $f\colon X\rightarrow Y$ in $G\sset$
is a \emph{$G$-fibration} (resp. \emph{$G$-equivalence}) if
$f^H\colon X^H\rightarrow Y^H$ is a fibration (resp. weak equivalence)
in $\sset$ for all $H\leqslant G$.
A map is a \emph{$G$-cofibration} if it satisfies the left lifting
property with respect to all acyclic $G$-fibrations.
With these notions of weak equivalences, fibrations and cofibrations,
$G\sset$ is a cofibrantly
generated proper model category.
In fact,
\begin{gather*}
I = \{i\colon (G/H\times \partial\Delta^{n})_+\rightarrow (G/H\times \Delta^n)_+|\ n\geq 0,\ H\leqslant G \}
\end{gather*}
and
\begin{gather*}
J = \{j\colon (G/H\times \Delta^n)_+\rightarrow (G/H \times \Delta^n\times \Delta^1)_+|\ n\geq 0,\ H\leqslant G\}
\end{gather*}
are sets of generating cofibrations and acyclic cofibrations.

More generally, if $\mathcal{F}$ is a family of subgroups, by which we mean a
collection of subgroups closed under conjugation
and taking subgroups,
there is a model structure relative to $\mathcal{F}$.
A map $f\colon X\rightarrow Y$ in $G\sset$ is an \emph{$\mathcal{F}$-fibration}
(resp. \emph{$\mathcal{F}$-equivalence}) if
$f^H\colon X^H\rightarrow Y^H$ is a fibration (resp. weak equivalence) in
$G\sset$ for all $H\in \mathcal{F}$.
This yields a cofibrantly generated proper model structure (cf.~\cite[IV. Theorem 6.5]{MM}).
As set of generating cofibrations (resp. acyclic cofibrations) we only take those
maps in $I$ (resp. $J$), where source and target have isotropy in $\mathcal{F}$.
Note that if $\mathcal{F}= \mathcal{ALL}$, this reproduces the model
structure introduced first and for arbitrary $\mathcal{F}$ the identity functor is a left Quillen functor
from the $\mathcal{F}$-model structure to the $\mathcal{ALL}$-model structure,
since every $G$-equivalence (resp. $G$-fibration) is an $\mathcal{F}$-equivalence
(resp. $\mathcal{F}$-fibration).
We want to point out that a map $A\rightarrow B$ is a cofibration in the $\mathcal{F}$-model
structure if and only if it is a cofibration in the $\mathcal{ALL}$-model structure and all
simplices not in the image have isotropy in $\mathcal{F}$.

If $\mathcal{F}$ is a family of subgroups of $G$, there is a
\emph{mixed} model structure on $G\sset$. The weak equivalences in this model
structure are the $\mathcal{F}$-equivalences and the cofibrations are the
$G$-cofibrations. The fibrations are defined by the appropriate lifting property
and are called \emph{mixed} $G$-fibrations.

\begin{ex}
\label{ex: Families Gn}
If the goup in question is $G\times \Sigma_n$ for $G$ an arbitrary finite
group and $\Sigma_n$ the symmetric group on $n$ letters, there is a particularly
important family denoted by $\mathcal{G}_n$.
It consists of all subgroups $J\leqslant G\times \Sigma_n$ such
that $J\cap \{1\}\times \Sigma_n = \{(1, 1)\}$.
Equivalently, those are the subgroups
of the form $\{(h, \rho(h))|\ h\in H\}$ where $H$ is a subgroup of $G$ and $\rho\colon H\rightarrow \Sigma_n$
is a homomorphism.
\end{ex}
\end{subsection}

\begin{subsection}{Equivariant enrichments of model categories}
\label{subsec: Equivariant enrichments of model categories}
We will now introduce the equivariant analogon of simplicial model categories.
The difference is that in the equivariant setting there are usually function
$G$-spaces, as opposed to simply function spaces.
Let $\mathcal{C}_G$ be the category
of $G$-objects in some category enriched over $\sset$.
Then $\mathcal{C}_G$ is enriched over $G\sset$, where we
equip the mapping space $\Map_{\mathcal{C}}(A, B)$
with the conjugation action.
It follows formally that passing to $G$-fixed points on mapping spaces yields a
category $G\mathcal{C}$
enriched over $\sset$ with objects the $G$-objects and spaces of equivariant maps.

Assume now in addition that $G\mathcal{C}$ has a model structure and
$\mathcal{C}_G$ is tensored and cotensored over
$\sset_{G}$.
The latter means
that there are functors
\begin{gather*}
\xymatrix@R=0cm{
 \sset_{G}\times \mathcal{C}_G\ar[r] & \mathcal{C}_G,\ (X, C) \ar@{|->}[r] & X\otimes C,
}\\
\xymatrix{
 \sset_{G}\times \mathcal{C}_G\ar[r] & \mathcal{C}_G,\ (X, D)\ar@{|->}[r] & \map_{\mathcal{C}}(X, D)
}
\end{gather*}
together with natural associativity isomorphisms and natural $G$-isomorphisms
\begin{gather*}
\xymatrix{
\Map_{\mathcal{C}}(X\otimes C, D)\cong \Map_{\sset}(X, \Map_{\mathcal{C}}(C, D))\cong \Map_{\mathcal{C}}(C, \map_{\mathcal{C}}(X, D).
}\end{gather*}

Passing to $G$-fixed points shows that $G\mathcal{C}$ is automatically
tensored and cotensored over $\sset$, though in general not over $G\sset$.

Given two maps $i\colon A\rightarrow X$ and $p\colon E\rightarrow B$ in
$G\mathcal{C}$, there is a $G$-map
\begin{equation*}
\entrymodifiers={+!! <0pt, \fontdimen22\textfont2>}
\xymatrix{
 \Map_{\mathcal{C}}(i^*, p_*)\colon \Map_{\mathcal{C}}(X, E)\ar[r] & \Map_{\mathcal{C}}(A, E)\times_{\Map_{\mathcal{C}}(A, B)} \Map_{\mathcal{C}}(X, B)
}
\end{equation*}
and we have
\begin{Def}
\label{def: G-simplicial model category}
In the situation above,
$G\mathcal{C}$ is called \emph{$G$-simplicial} if for all cofibrations $i$ and
all fibrations $p$ the map $\Map_{\mathcal{C}}(i^*, p_*)$ is a $G$-fibration,
which is in addition a $G$-equivalence if $i$ or $p$ is. 
\end{Def}
\begin{ex}
The most elementary example is, of course, $G\sset$ itself.
\end{ex}

\begin{Lemma}The following are equivalent.
\begin{itemize}
\item[(a)] $G\mathcal{C}$ is $G$-simplicial.
\item[(b)] For all cofibrations $f\colon A\rightarrow X$ in $G\mathcal{C}$
           and all $G$-cofibrations $i: K\rightarrow L$ in $G\sset$
      the pushout product map 
      \[
\entrymodifiers={+!! <0pt, \fontdimen22\textfont2>}
      \xymatrix{
      i\Box f\colon K\otimes X\cup_{K\otimes A} L\otimes A \ar[r] & L\otimes X 
      }
      \]
      is a cofibration, which is in addition acyclic if $f$ or $i$ is.
\item[(c)] For all fibrations $p\colon E\rightarrow B$ in $G\mathcal{C}$ and all $G$-cofibrations $i\colon K\rightarrow L$
      in $G\sset$
      the map
      \[
\entrymodifiers={+!! <0pt, \fontdimen22\textfont2>}
      \xymatrix{
      \map_{\mathcal{C}}(i^*, p_*)\colon\map_{\mathcal{C}}(L, E)\ar[r] & \map_{\mathcal{C}}(K, E)\times _{\map_{\mathcal{C}}(K, B)} \map_{\mathcal{C}}(L, B)
      }
      \]
      is a fibration, which is in addition acyclic if $i$ or $p$ is.
\end{itemize}
\end{Lemma}
\begin{proof}
 The proof is formal and thus skipped.
\end{proof}

We also have
\begin{Lemma}
\label{lem: Detection of weak equivalences}
Suppose $G\mathcal{C}$ is $G$-simplicial. A map $f\colon A\rightarrow B$ between cofibrant objects
is a weak equivalence if and only if for all fibrant objects $X$ the induced map
\[
\Map_{\mathcal{C}}(f^*, X)\colon \Map_{\mathcal{C}}(B, X)\rightarrow \Map_{\mathcal{C}}(A, X)
\]
is a $G$-equivalence.
\end{Lemma}
\begin{proof}
 See~\cite[Lemma 4.5]{BF} for instance.
\end{proof}

\end{subsection}

\end{section}

\newpage
\begin{section}{Recollections on \texorpdfstring{$G$-symmetric spectra}{G-symmetric spectra}}
\label{sec: Recollections on G symmetric spectra}
From now on, $G$ denotes a fixed finite group. Based on~\cite{Hausmann}, we will briefly
recall the basic definitions concerning $G$-symmetric spectra and the two model structures
we will work with later on.

\begin{subsection}{$G$-symmetric spectra}
\label{subsec: G symmetric spectra}

If $M$ is a finite set, 
the $M$-fold smash product $S^M = \wedge_M S^1$
of the (based) simplicial circle $S^1 = \Delta^1/\partial \Delta^1$
is the quotient of the $M$-fold product of $S^1$ by the $M$-fold
wedge product based at $\vee_M S^1/\vee_M S^1$.
The set $\{1, \ldots, n\}$ will be denoted by ${\bf{n}}$.

\begin{Def}
\label{def: Symmetric spectrum}
A \emph{symmetric spectrum} $X$ consists of
\begin{itemize}
\item[(a)] for all $n\geq 0$, a based $\Sigma_n$-space $X_n$ and
\item[(b)] for all $n\geq 0$, a based structure map $\sigma_n\colon X_n\wedge S^1\rightarrow X_{n+1}$.
\end{itemize}
This is subject to the condition that for all $n$, $m\in\N$, the \emph{iterated} structure map
\[
\xymatrix{
\sigma_n^m\colon X_n\wedge S^{{\bf{m}}} \cong (X_n\wedge S^1)\wedge S^{{\bf{m-1}}}\ar[r] & X_{n+1}\wedge S^{{\bf{m-1}}}\ar[r] & \ldots \ar[r] & X_{n+m}
}
\]
is $\Sigma_n\times \Sigma_m$-equivariant.
A morphism of symmetric spectra $f\colon X\rightarrow Y$ is a sequence of based $\Sigma_n$-maps $f_n\colon X_n\rightarrow Y_n$
such that, for all $n\in\N$, the square
\[
\xymatrix{
X_n\wedge S^1 \ar[r]^{f_n\wedge S^1} \ar[d]_{\sigma_n^X} & Y_n\wedge S^1 \ar[d]^{\sigma_n^Y}\\
X_{n+1}       \ar[r]_{f_{n+1}}        & Y_{n+1}
}
\]
commutes. The category of symmetric spectra will be denoted by ${Sp}^\Sigma$.
\end{Def}

\begin{Def}
\label{def: G symmetric spectrum}
A \emph{$G$-symmetric spectrum} is a $G$-object in $Sp^\Sigma$.
A morphism between $G$-symmetric spectra is a morphism of symmetric spectra
commuting with the $G$-action.
The category of $G$-symmetric spectra will be denoted by $G{Sp}^{\Sigma}$.
\end{Def}

Let $M$ be a finite $G$-set of order $m$.
We endow the space $S^M$ with the $G$-action
\[
g \cdot (\wedge_{l\in M} x_l) := \wedge_{l\in M} x_{g^{-1}l}.
\]
The set of bijections $\Bij(m,M)$
carries a left $G$-action by postcomposition and a right $\Sigma_m$-action by
precomposition.
The value of a $G$-symmetric spectrum $X$ at $M$ is defined to be the $G$-space
\[
X(M) := X_m\wedge_{\Sigma_m} \Bij(m, M)^+,
\]
where one identfies $x\wedge f\sigma$ with $\sigma_*(x)\wedge f$ whenever $\sigma\in \Sigma_m$ and $G$ acts diagonally.
The $G\times \Sigma_n$-space $X({\bf{n}})$ is naturally isomorphic to $X_n$, by sending $[x\wedge \phi]$ to $\phi_*(x)$.
In Section \ref{subsec: A characterization of flat cofibrations} below, it turns out to be more convenient to work with the values at ${\bf{n}}$,
so we adopt this point of view from now on.

Given two finite $G$-sets $M$ and $N$,
there is a generalized structure map
\[
\xymatrix{
\sigma_M^N\colon X(M)\wedge S^N \ar[r] & X(M\sqcup N).
}
\]
Chosen any isomorphism $\psi\colon n\rightarrow N$ ($\sigma_M^N$ does not depend
on this choice), it is given by
\[
([x\wedge f]\wedge s)\mapsto [\sigma_m^n(x\wedge \psi_*^{-1}(s))\wedge (f\sqcup \psi)].
\]
The morphism spaces of symmetric spectra equipped with the conjugation action
gives rise to the enrichment of $GSp^\Sigma$ over $G\sset$ as explained
in Section \ref{subsec: Equivariant enrichments of model categories}.
\end{subsection}

\begin{subsection}{The flat level model structure}
\label{subsec: The flat level model structure}

Now, we introduce the \emph{flat level model structure} on $GSp^\Sigma$.
For the definition of the latching objects we refer the reader to Section \ref{subsec: A characterization of flat cofibrations}.
Recall the definitions of the families $\mathcal{G}_n$ given in
Example \ref{ex: Families Gn} and the various model structures on $G\times \Sigma_n$-spaces introduced in
Section \ref{subsec: Model structures on G spaces}.

A morphism $f\colon X\rightarrow Y$ of $G$-symmetric spectra is a \emph{$G$-level equivalence} (resp. \emph{$G$-level fibration}) if, for all
$n\in \N$, the $G\times \Sigma_n$-map $f({\bf{n}})\colon X({\bf{n}})\rightarrow Y({\bf{n}})$ is a $\mathcal{G}_n$-equivalence (resp. mixed $G\times \Sigma_n$-fibration).
The morphism $f$ is a \emph{$G$-flat cofibration} if, for all $n\in \N$,
the pushout product map
\[
\xymatrix{
\nu_n(f)\colon X({\bf{n}})\cup_{L_n(X)} L_n(Y) \ar[r] & Y({\bf{n}})
}
\]
is a $G\times \Sigma_n$-cofibration.

\begin{Prop}[Flat level model structure]
\label{prop: Flat level model structure}
The classes of $G$-flat cofibrations, $G$-level fibrations and $G$-level
equivalences define a proper $G$-simplicial model
structure on the category of $G$-symmetric spectra.
\end{Prop}
\begin{proof}
This is~\cite[Corollary 2.8.9, Proposition 2.8.11]{Hausmann}.
\end{proof}
\end{subsection}

\begin{subsection}{The stable flat model structure and $\pi_*$-isomorphisms}
\label{subsec: The stable flat model structure and pi star isomorphisms}

We will also need the \emph{stable flat model structure} on $G$-symmetric spectra.

\begin{Def}
\label{def: G levelwise Kan}
 A $G$-symmetric spectrum $X$ is \emph{$G$-levelwise Kan} if, for all
 subgroups $H\leqslant G$ and all finite $H$-sets $M$, the space
 $X(M)^H$ is Kan.
\end{Def}

\begin{Def}
\label{def: G Omega spectrum}
 A $G$-symmetric spectrum $X$ is called \emph{$G\Omega$-spectrum} if it is $G$-levelwise Kan and, for all subgroups $H\leqslant G$ and
 all finite $H$-sets $M$ and $N$, the adjoint of the generalized structure map
 \[
  \xymatrix{
  X(M)\ar[r] & \Map_{\sset}(S^N, X(M\sqcup N))
  }
 \]
is an $H$-equivalence.
\end{Def}

\begin{Def}
 \label{def: G stable equivalence}
A map $X\rightarrow Y$ of $G$-symmetric spectra is a \emph{$G$-stable equivalence} if  for all $G$-level fibrant $G\Omega$-spectra $Z$ and
for some $G$-flat replacement $X^c\rightarrow Y^c$
the induced map
\[
\xymatrix{
[Y^c, Z]^G \ar[r] & [X^c, Z]^G
}
\]
on based $G$-homotopy classes of $G$-maps is a bijection.
\end{Def}

A \emph{$G$-stable fibration} is a map that satisfies the right lifting property
with respect to all $G$-flat cofibrations that are $G$-stable equivalences.

\begin{Thm}[Stable flat model structure]
\label{thm: Stable flat model structure}
The classes of $G$-flat cofibrations, $G$-stable equivalences and $G$-stable fibrations
define a proper $G$-simplicial model category structure on the category
of $G$-symmetric spectra. The fibrant objects are precisely the
$G$-level fibrant $G\Omega$-spectra.
\end{Thm}
\begin{proof}
\cite[Theorem 2.12.5, Proposition 2.12.6]{Hausmann}.
\end{proof}

An important fact is that $\pi_*$-isomorphisms as defined below are
$G$-stable equiva-lences (cf.~\cite[Theorem 2.10.19]{Hausmann}).
To this end, let $\mathcal{U}$ be a complete $G$-set universe, that is a countably
infinite $G$-set with the property that any finite $G$-set embeds
infinitely often disjointly in it.
We denote by $s(\mathcal{U})$ the set of all finite $G$-subsets of $\mathcal{U}$,
partially ordered by inclusion.
For all $n\geq 0$, a $G$-symmetric spectrum $X$ gives rise to a functor
\[
\xymatrix{
s(\mathcal{U})\ar[r] & \text{Sets},\ M\mapsto [|S^{{\bf{n}}\sqcup M}|, |X(M)|]_*^G,
}
\]
where $|-|$ denotes geometric realization.
Here, an inclusion $M\subset N$ sends a map $f\colon |S^{{\bf{n}}\sqcup M}|\rightarrow |X(M)|$ to the composition
\[
\resizebox{12,5cm}{!}{
\xymatrix{
|S^{{\bf{n}}\sqcup N}|\cong |S^{{\bf{n}}\sqcup M}|\wedge |S^{N-M}|\ar[r]& |X(M)|\wedge |S^{N-M}|\cong |X(M)\wedge S^{N-M}|\ar[r]& |X(N)|,
}
}
\]
where the last map is the geometric realization of the generalized structure map $\sigma_M^{N-M}$.
For $n\geq 0$, we define
\[
 \pi_n^{G, \mathcal{U}}(X) := \colim_{M\in s(\mathcal{U})}[|S^{{\bf{n}}\sqcup M}|, |X(M)|]_*^G.
\]
In order to define the negative homotopy groups, for a $G$-symmetric spectrum $X$ and a finite $G$-set $M$,
we define the shift $sh^MX$ by
\[
(sh^MX)({\bf{n}}) : = X(M\sqcup {\bf{n}}),
\]
where the structure maps are induced by the structure maps of $X$
(cf.~\cite[Definition 2.3.2]{Hausmann}).
Then, for $n<0$, we set
\[
\pi_n^{G, \mathcal{U}}(X) := \pi_0^{G, \mathcal{U}}(sh^{\bf{-n}}X).
\]

\begin{Def}
\label{def: pi * isomorphism}
 A map $f\colon X\rightarrow Y$ of $G$-symmetric spectra is a \emph{$\pi_*$-isomorphism}
 if for all subgroups $H\leqslant G$ and all $n\in\Z$ and some (hence any)
 complete $G$-set universe $\mathcal{U}$ the map
 \[
  \xymatrix{
  \pi_n^{H,\mathcal{U}}f\colon \pi_n^{H,\mathcal{U}}(X)\ar[r] & \pi_n^{H,\mathcal{U}}(Y)
  }
 \]
is an isomorphism, where $X$ and $Y$ are considered as $H$-symmetric spectra via restriction and $\mathcal{U}$ is considered as a
complete $H$-set universe via restriction.
\end{Def}

The relevance of the flat model structure is that it is compatible with the smash product.
Here we we endow the smash product $X\wedge Y$ of two $G$-symmetric spectra
with the diagonal $G$-action.

\begin{Prop}
\label{prop: Smashing with flat G symmetric spectrum preserves level equivalences and pi * isomorphisms}
If a $G$-symmetric spectrum $X$ is $G$-flat then $X\wedge - $ preserves (acyclic) $G$-flat cofibrations, $G$-level equivalences and $\pi_*$-isomorphisms.
\end{Prop}
\begin{proof}
For the first three statements one uses the skeleton filtration of a $G$-flat spectrum $X$ and induction.
The fact the $X\wedge -$ preserves $\pi_*$-isomorphisms can be proven as in~\cite[Proposition 2.3.29]{Stolz}.
\end{proof}

\end{subsection}
\end{section}

\newpage
\begin{section}{\texorpdfstring{$\Gamma$-$G$-spaces}{Gamma-G-spaces} and two level model structures}
\label{sec: Equivariant Gamma spaces and two level model structures}
In this section we introduce $\Gamma$-$G$-spaces and the projective
and strict level model structures.

\begin{subsection}{Generalities on $\Gamma$-$G$-spaces}
\label{subsec: Generalities on Gamma G spaces}

\begin{Def}
\label{def: Gamma}
 We define $\Gamma$ to be the category with objects the based finite sets $n^+ = \{1, \ldots, n\}\sqcup \{+\}$
 based at $\{+\}$ together with basepoint preserving maps.
\end{Def}
\begin{Def}
 \label{def: Gamma G space}
A $\Gamma$-$G$-space is a functor $A\colon \Gamma\rightarrow G\sset$,
such that $A(0^+) = *$. A map of $\Gamma$-$G$-spaces is a natural transformation of functors.
The category of $\Gamma$-$G$-spaces will be denoted by $\Gamma(G\sset)$.
\end{Def}

Given a $\Gamma$-$G$-space $A$, the value $A(X)$ at a based $G$-space $X$ is defined by
\begin{gather*}
 A(X) = \bigsqcup_{n^+\in \Gamma}\ \Map_{\sset}(n^+, X)\times A(n^+)\ /\sim,
\end{gather*}
where we divide out the equivalence relation generated by $(\phi^*f, a)\sim (f, A(\phi)(a))$
for $f\colon n^+ \rightarrow X$, $\phi\colon k^+\rightarrow n^+$ and $a\in A(k^+)$.
This space is based at the equivalence class of $(*, *)$.

\begin{Rmk}
It seems natural to define equivariant $\Gamma$-spaces as equivariant functors
from the category $\Gamma_G$ of finite based $G$-sets with based maps to the category $\sset_G$,
where $G$ acts on the morphisms of the categories by conjugation.
But Shimakawa observed in~\cite[Theorem 1]{Shimakawa Note} that this
gives a category equivalent to $\Gamma(G\sset)$. The reason is precisely that the values
of an equivariant functor $A\colon \Gamma_G\rightarrow \sset_G$ on a based finite $G$-set $S^+$ can be recovered
by evaluating the underlying $\Gamma$-space on $S^+$.
\end{Rmk}

If $A$ is a based $G$-space and $X$ is an object of $\Gamma(G\sset)$, then, in level $n^+$, $A\wedge X$
and $\map_{\Gamma(\sset)}(A, X)$ are given by $A\wedge X(n^+)$ with diagonal
action and $\Map_{\sset}(A, X(n^+))$ with conjugation action respectively.
The enrichement of $\Gamma$-$G$-spaces in $G$-spaces is constructed as follows. Suppose $A$ and $B$ are $\Gamma$-$G$-spaces.
$\Map_{\Gamma(\sset)}(A, B)$ is the space with $n$-simplices consisting of the (not necessarily equivariant)
natural transformations $(\Delta^i)^+\wedge A \rightarrow B$ endowed with the conjugation action.
Passing to fixed points yields a space
$\Map_{\Gamma(\sset)}(A, B)^G$ with $0$-simplices $\Gamma(G\sset)(A, B)$.

For a $\Gamma$-$G$-space $A$ and an arbitrary based finite $G$-set $S^+$ we have a $G$-map
\begin{gather*}
\xymatrix{
 P_{S^+}\colon A(S^+)\ar[r] & A(1^+)^{S^+} = \Map_{\sset}(S^+, A(1^+)),
}
\end{gather*}
given by $(P_{S^+}(a))(s) = A(p_s)(a)$, where $p_s\colon S^+\rightarrow 1^+$ maps $s$ to $1$ and
everything else to the basepoint.

\begin{Def}
\label{def: special Gamma G space}
A $\Gamma$-$G$-space $A$ is called \emph{special} if
$P_{S^+}$ is a $G$-equivalence for all based finite $G$-sets $S^+$. 
\end{Def}

Let $A$ be special and define the fold map $\nabla\colon 2^+\rightarrow 1^+$ to be
the map sending $1$ and $2$ to $1$. The zigzag
\begin{gather*}
\xymatrix{
 A(1^+)\times A(1^+) &A(2^+) \ar[l]^-{\simeq_G}_-{p_1\times p_2} \ar[r]^-{\nabla} &A(1^+)
}
\end{gather*}
induces the structure of a commutative monoid on the set of path components of the $H$-fixed points
$\pi_0^H(A(1^+))$.

\begin{Def}
A special $\Gamma$-$G$-space $A$ is called \emph{very} \emph{special}
if, in addition, $A(1^+)$ is grouplike. That is, $\pi_0^H(A(1^+))$ with the composition law just defined is a group
for all subgroups $H\leqslant G$.
\end{Def}
\end{subsection}

\begin{subsection}{Level model structures}
Now we introduce two level model structures on the category of $\Gamma$-$G$-spaces.

\begin{subsubsection}{The projective model structure}
\label{subsubsec: The projective model structure}

There is the projective model structure on $\Gamma(G\sset)$, which was also employed in~\cite{Santhanam}.
Recall the sets $I$ and $J$ defined in Section \ref{subsec: Model structures on G spaces}.
A morphism of $\Gamma$-$G$-spaces $f\colon A\rightarrow B$ is called
\emph{level equivalence}
(resp. \emph{level fibration}) if
$f(S^+)$ is a  $G$-equivalence (resp. $G$-fibration) for all based
finite $G$-sets $S^+$.
A \emph{projective cofibration} is a map which satisfies the
left lifting property with respect to all level fibrations which are in addition level equivalences.

We define $\Gamma I$ and $\Gamma J$ to be the sets
consisting of the maps $i\wedge\Gamma_G(S^+, -)$ and $j\wedge\Gamma_G(S^+, -)$
respectively, where $S^+$ runs over a set of representatives of
isomorphism classes of based finite $G$-sets and $i\in I$ and $j\in J$ respectively.

The proof of the following result is standard.
\begin{Thm}[Projective model structure]
\label{thm: projective model structure}
The classes of level fibrations,
level equivalences and projective cofibrations define a cofibrantly generated proper $G$-simplicial
model category structure on the category of $\Gamma$-$G$-spaces. The sets $\Gamma I$ and $\Gamma J$ can be taken as sets of
generating cofibrations and generating acyclic cofibrations respectively.
\end{Thm}
\end{subsubsection}

\begin{subsubsection}{The strict model structure}
\label{subsubsec: The strict model structure}
We will now introduce a model structure which reduces to
the model structure due to
Bousfield and Friedlander (cf.~\cite{BF}) if $G$ is the trivial group.
In order to introduce this model structure we need
\begin{Def}
 \label{def: skn cskn}
The $n$th \emph{skeleton} of a $\Gamma$-$G$-space $A$ is the
$\Gamma$-$G$-space given by
\begin{equation*}
 (\sk_nA)(m^+) := \colim_{k^+\rightarrow m^+,\ k\leq n}\ A(k^+)
\end{equation*}
in level $m^+$.
Dually, the $n$th \emph{coskeleton} of $A$ is defined to be the $\Gamma$-$G$-space
given by
\begin{equation*}
 (\csk_nA)(m^+) := \lim_{m^+\rightarrow j^+,\ j\leq n}\ A(j^+).
\end{equation*}
\end{Def}
\begin{Rmk}
More conceptual definitions of these two functors appear in Section \ref{subsec: The strict model structure for Gamma G spaces},
where they occur naturally when maps between $\Gamma$-$G$-spaces are constructed inductively.
\end{Rmk}

There are natural maps
$(\sk_n A)\rightarrow A \rightarrow (\csk_n A)$.
Hence, a map $f\colon X\rightarrow Y$ of $\Gamma$-$G$-spaces induces, for all $n\geq 0$,
maps
\begin{gather*}
\xymatrix{
i_n(f)\colon (\sk_{n-1}Y)(n^+)\cup_{(\sk_{n-1} X)(n^+)} X(n^+)\ar[r] & Y(n^+)
}
\end{gather*}
and
\begin{gather*}
\xymatrix{
p_n(f)\colon X(n^+)\ar[r] & (\csk_{n-1}X)(n^+)\times_{(\csk_{n-1}Y)(n^+)}Y(n^+).
}
\end{gather*}

Then $f$ is called \emph{strict cofibration} (resp. \emph{strict fibration}) if, for all $n\geq 0$, the map $i_n(f)$
(resp. $p_n(f)$) is a $\mathcal{G}_n$-cofibration (resp. $\mathcal{G}_n$-fibration).
The map $f$ is called \emph{strict equivalence} if it is levelwise a $\mathcal{G}_n$-equivalence.

\begin{Rmk}
Note, for any map of $\Gamma$-$G$-spaces $f\colon A\rightarrow B$, being
a strict equivalence amounts to saying that for any based finite $H$-set $S^+$,
$f(S^+)$ is an $H$-equivalence. Yet, it turns out that this is a little bit redundant.
Indeed, assume the seemingly weaker condition that $f(S^+)$ is a $G$-equivalence for all pointed
$G$-sets $S^+$ and let $T^+$ be any based finite $H$-set. But $T^+$ is an $H$-retract of a $G$-set $Q^+$ and $H$-equivalences
are closed under retracts, hence $f(T^+)$ is an $H$-equivalence.
In particular, the strict equivalences coincide with the level equivalences defined in Section \ref{subsubsec: The projective model structure}.
\end{Rmk}
Let $\mathcal{G}^1_n\sset$ be the category of $G\times \Sigma_n$-spaces with the $\mathcal{G}_n$-model structure 
and let $\mathcal{G}^2_n\sset := (G\times \Sigma_n)\sset$ be the category of $G\times \Sigma_n$-spaces with the $\mathcal{ALL}$-model structure.
The assumptions of Theorem \ref{thm: gen strict model structure} are
satisfied. Hence we have
\begin{Thm}[Strict model structure]
\label{thm: strict model structure}
The classes of strict equivalences, strict fibrations and strict cofibrations
define a proper $G$-simplicial model category structure on the category of $\Gamma$-$G$-spaces.
\end{Thm}
\end{subsubsection}

\begin{subsubsection}{Comparison of the projective and the strict model structure}
\label{subsubsec: Comparison of the projective and the strict model structure}

\begin{Prop}
\label{prop: quillen equivalence projective and stict model structure}
The identity functor induces a Quillen equivalence
\begin{gather*}
\xymatrix{
 \id: \Gamma(G\sset)^{\text{projective}} \ar@<.3 ex>[r] &  \Gamma(G\sset)^{\text{strict}}: \id. \ar@<.3 ex>[l]
}
\end{gather*}
\end{Prop}
\begin{proof}
We have already seen that the weak equivalences coincide.
Hence, we only have to check that the generating (acyclic) cofibrations
are indeed (acyclic) strict cofibrations.
In fact, if $i\colon A\rightarrow B$ is a $G$-cofibration of $G$-spaces and $S^+$ 
is any based finite $G$-set, then all simplices
not in the image of $i_n(i\wedge \Gamma_G(S^+, -))$ have isotropy contained in
$\mathcal{G}_n$, since the set $(\sk_{n-1}\Gamma_G(S^+, -))(n^+)$
consists precisely of the non-surjective maps.
\end{proof}
\end{subsubsection}

At last, we want to mention an important lemma about
the coend $A(X)$ 
of $A\in\Gamma(G\sset)$ and $X\in G\sset$.

\begin{Lemma}
\label{lem: diag}
 Suppose $f\colon A\rightarrow B$ in $\Gamma(G\sset)$ is a level equivalence,
 then, for any based $G$-space $X$, $f(X)\colon A(X) \rightarrow B(X)$
 is a $G$-equivalence.
\end{Lemma}
\begin{proof}
 First of all, we observe that $A(X)$ is just the diagonal of the
 bisimplicial set $B_{n, m} = A(X_n)_m$. Since taking fixed points commutes with taking
 the diagonal, it suffices to show that $A(X_n)_*^H \rightarrow B(X_n)_*^H$ is an
 ordinary weak equivalence for all subgroups $H\leqslant G$ (cf.~\cite[Theorem B.2]{BF}). This holds true by assumption if all $X_n$
 are finite.
 Moreover, if $S^+\rightarrow T^+$ is an injective map of based finite $G$-sets, then $A(S^+)^H\rightarrow A(T^+)^H$
 is injective, too, hence is a cofibration upon geometric realization.
 The assertion now follows, because any based $G$-set is the filtered colimit of its based finite $G$-subsets and
 homotopy groups commute with filtered colimits along cofibrations.
\end{proof}

\end{subsection}
\end{section}

\newpage
\begin{section}{Unstable comparison of \texorpdfstring{$\Gamma$-$G$-spaces}{Gamma-G-spaces} and \texorpdfstring{$G$-symmetric spectra}{G-symmetric spectra}}
\label{sec: Gamma G spaces and G symmetric spectra}
A $\Gamma$-$G$-space $A$ gives rise to a $G$-symmetric spectrum $A(\Sp)$.
We show in Section \ref{subsec: Quillen pairs between Gamma G spaces and G symmetric spectra} how this construction yields
Quillen pairs between the strict and projective model structures on $\Gamma$-$G$-spaces and the flat level model structure on $G$-symmetric spectra.
In Section \ref{subsec: Some properties of spectra of the form Sp otimes A}, we show that the spectra obtained from
$\Gamma$-$G$-spaces are connective
and that very special $\Gamma$-$G$-spaces yield $G\Omega$-spectra
upon a level fibrant replacement.
Finally, in the last subsection we compare suitable subcategories of the homotopy
categories of $\Gamma$-$G$-spaces and $G$-symmetric spectra with respect to the level model structures.

\begin{subsection}{Quillen pairs between $\Gamma$-$G$-spaces and $G$-symmetric spectra}
\label{subsec: Quillen pairs between Gamma G spaces and G symmetric spectra}

Let $A$ be a $\Gamma$-$G$-space.
The $G$-symmetric spectrum $A(\Sp)$ is given by $A(S^{{\bf{n}}})$ in level $n$.
Here $G$ acts on $A$ and $\Sigma_n$ acts by permuting the sphere coordinates.
The structure map $\sigma_n\colon A(S^{\bf{n}})\wedge S^1\rightarrow A(S^{{\bf{n+1}}})$
is defined by sending a class $[(v_1, \ldots, v_n), a]\wedge w$ to the class $[(v_1\wedge w, \ldots, v_n\wedge w), a]$.

Conversely, given a $G$-symmetric spectrum, we may construct a $\Gamma$-$G$-space
denoted $\Phi(\Sp, X)$ by setting
$\Phi(\Sp, X)(n^+) := \Map_{Sp^\Sigma}(\Sp^{\times n}, X)$.
With these definitions, we have
\begin{Prop}
\label{prop: adjunction gamma spaces spectra}
The functors
\begin{gather*}
\xymatrix{
 (-)(\Sp) : \Gamma(G\sset) \ar@<.3 ex>[r] &  {GSp^\Sigma}: \Phi(\Sp, - ) \ar@<.3 ex>[l]
}
\end{gather*}
form a Quillen pair between the strict model structure on $\Gamma$-$G$-spaces
and the flat level model structure on $G$-symmetric spectra.
\end{Prop}
\begin{proof}
It well-known that this is an adjunction (cf.~\cite[Proposition 3.3]{Segal 1},~\cite[Lemma 4.6]{BF}).
We prove that $\Phi(\Sp,-)$ sends $G$-level fibrations (resp. acyclic $G$-level fibrations) to
strict fibrations (resp. acyclic strict fibrations).
Note that, for any $G$-symmetric spectrum $X$, we have
\begin{equation*}
(\csk_{n-1} \Phi(\Sp, X))(n^+) = \nlim_{n^+\rightarrow j^+,\ j\leq n-1} \Map_{Sp^\Sigma}(\Sp^{\times j}, X)
\cong \Map_{{Sp^\Sigma}}(\Sp^{\times n}_{\leq n-1}, X),
\end{equation*}
where
\begin{equation*}
(\Sp^{\times n}_{\leq n-1})_m = \{(x_1, \ldots, x_n)\in (S^m)^{\times n}|\ x_i = x_j\ \text{for some}\ i\neq j\ \text{or}\ x_i = *\ \text{for some}\ i\},
\end{equation*}
as a direct computation of colimits proves.
So we have to show that the map
\begin{equation*}
\entrymodifiers={+!! <0pt, \fontdimen22\textfont2>}
\xymatrix{
 \Map_{Sp^\Sigma}(\Sp^{\times n}, X)\ar[r] & \Map_{Sp^\Sigma}(\Sp^{\times n}_{\leq n-1}, X)\times_{\Map_{Sp^\Sigma}(\Sp^{\times n}_{\leq n-1}, Y)} \Map_{Sp^\Sigma}(\Sp^{\times n}, Y),
}
\end{equation*}
induced by the inclusion $\Sp^{\times n}_{\leq n-1}\rightarrow \Sp^{\times n}$
and a $G$-level fibration (resp. acyclic $G$-level fibration) $f\colon X\rightarrow Y$
is a $\mathcal{G}_n$-fibration (resp. acyclic $\mathcal{G}_n$-fibration).

Equivalently, for all $n$ and all based finite $H$-sets $S^+$ of order $n+1$, the map
\begin{gather*}
\entrymodifiers={+!! <0pt, \fontdimen22\textfont2>}
\xymatrix{
 \Map_{Sp^\Sigma}(\Sp^{\times S}, X)\ar[r] &
 \Map_{Sp^\Sigma}(\Sp^{\times S}_{\leq n-1}, X)\times_{\Map_{Sp^\Sigma}(\Sp^{\times S}_{\leq n-1}, Y)} \Map_{Sp^\Sigma}(\Sp^{\times S}, Y),
}
\end{gather*}
is an $H$-fibration.
Since $H$-symmetric spectra are $H$-simplicial, it is therefore sufficient to show
that $\Sp^{\times S}_{\leq n-1}\rightarrow \Sp^{\times S}$ is an $H$-flat cofibration
of $H$-symmetric spectra.
This is the content of Proposition \ref{prop: product G flat} in the appendix.
\end{proof}

Together with Proposition \ref{prop: quillen equivalence projective and stict model structure}
this implies
\begin{Prop}
The functors
\[
\xymatrix{
 (-)(\Sp) : \Gamma(G\sset) \ar@<.3 ex>[r] &  GSp^\Sigma: \Phi(\Sp,-), \ar@<.3 ex>[l]
}
\]
form a Quillen pair between the projective model structure on $\Gamma$-$G$-spaces
and the flat level model structure on $G$-symmetric spectra.
\end{Prop}
\end{subsection}

\begin{subsection}{Some properties of $G$-spectra of the form $A(\Sp)$}
\label{subsec: Some properties of spectra of the form Sp otimes A}

\begin{subsubsection}{Connectivity}
\label{subsubsec: Connectivity}
We want to show that all the negative homotopy groups
of a spectrum arising from a $\Gamma$-$G$-space vanish.
This is accomplished by introducing a two sided bar construction.

Given a $\Gamma$-$G$-space $A$, we define a functor $\sigma A$ from $\Gamma$ to $G$-spaces by setting
\[
(\sigma A)(n^+) = B(n^+, \Gamma_G, A) = \mathrm{diag}(B_\bullet({n^+}, \Gamma_G, A)),
\]
where for a $G$-space $X$, $B_\bullet(X, \Gamma_G, A)$ denotes the simplicial
$G$-space with $k$-simplices
\[
\resizebox{12,5cm}{!}{
\xymatrix{
 B_k(X, \Gamma_G, A)  = \bigsqcup_{s_0^+, \ldots, s_k^+\in \Gamma_G} A(s_0^+)\times \Gamma_G(s_0^+, s_1^+) \times \ldots \times \Gamma_G(s_{k-1}^+, s_k^+)\times (X)^{s_k}.
}
}
\]
There is a natural $G$-isomorphism
$B(X, \Gamma_G, A)\rightarrow (\sigma A)(X)$
(cf.~\cite[Proof of Theorem 1.5]{Woolfson}).
Then $(\sigma A)/(\sigma A(0^+))$ is a $\Gamma$-$G$-space and the $n$th level of the
$G$-symmetric spectrum $(\sigma A)(\Sp)/(\sigma A(0^+))$ is $G$-isomorphic to
$B(S^{\bf{n}}, \Gamma_G, A)/B(*, \Gamma_G, A)$.

\begin{Lemma}
\label{lem: cofibrant approximation}
 The map
 \[
 \xymatrix{
 B_k(X, \Gamma_G, A)\ar[r]& A(X),\ (a, f_0, \ldots, f_{k-1}, \phi)\mapsto (\phi_*{f_{k-1}}_*\ldots {f_0}_*)(a),
}
 \]
 induces a natural level equivalence $(\sigma A)/(\sigma A(0^+))\rightarrow A$ of $\Gamma$-$G$-spaces.
 In particular, it induces a $G$-level equivalence of $G$-symmetric spectra
 \[
 \xymatrix{
 (\sigma A)(\Sp)/(\sigma A)(0^+)\ar[r] & A(\Sp).
}
 \]
\end{Lemma}
\begin{proof}
As in~\cite[Proposition 5.19]{Lydakis}, one shows that the map induces, for any based finite $G$-set $S^+$, a
$G$-equivalence $(\sigma A)(S^+)\rightarrow A(S^+)$. The first statement follows now, since $(\sigma A)(0^+)\rightarrow (\sigma A)(S^+)$ is a $G$-cofibration.
The second part follows from Lemma \ref{lem: diag}.
\end{proof}

\begin{Lemma}
\label{lem: connectivity}
Let $A$ be a $\Gamma$-$G$-space. If $X$ is any pointed $G$-space such that, for all $K\leqslant L$, $\conn(X^K)\geq \conn(X^L)\geq 1$
then, for all $H\leqslant G$, 
we have
\[
\conn(B(X, \Gamma_G, A)^H/B(x_0, \Gamma_G, A)^H)\geq \conn(X^H).
\]
\end{Lemma}
\begin{proof}
We check the connectivity of
\[
|B_\bullet(|X|, \Gamma_G, |A|)^H/B_\bullet(x_0, \Gamma_G, |A|)^H|.
\]
This space is the geometric realization of the simplicial space with $k$-simplices
\begin{gather*}
 \bigvee_{s_0^+, \ldots, s_k^+\in\Gamma_G} (|A(s_0^+)|^+)^H\wedge (\Gamma_G(s_0^+, s_1^+)^+)^H\wedge \ldots\wedge (\Gamma_G(s_{k-1}^+, s_k^+)^+)^H\wedge (|X|^{s_k})^H.
\end{gather*}
Now, each wedge summand is at least as connected as $(|X|^{s_k})^H$.
Writing $s_k^+\cong_H H/L_1^+\vee \ldots \vee H/L_j^+$, we find $(X^{s_k})^H\cong X^{L_1}\times \ldots \times X^{L_n}$,
which is, by our assumptions, firstly, at least as connected as $X^H$ and, secondly, simply connected.
Hence, all simplicial levels are so.
The space in question is a wedge of based $G$-$CW$-complexes of the form $|A(S^+)|^+\wedge |X|^{T}$ and the
degeneracy maps are just inclusions of certain wedge summands. One can now argue as in~\cite[Theorem 11.12]{Geometry of Iterated Loop Spaces}.
\end{proof}
\begin{Cor}
\label{cor: connectivity of Sp otimes A}
 The spectrum $A(\Sp)$ is connective for all $\Gamma$-$G$-spaces $A$, that is
 all negative homotopy groups vanish.
\end{Cor}
\begin{proof}
Fix $n\geq 1$. The $(-n)$th homotopy group with respect to $H$ is computed as a colimit over finite $H$-sets $M$
of $[|S^M|, |A(S^{M\sqcup \bf{n}})|]^H_*$. So we may assume $|M^H|\geq 1$.
For such $M$, we have, for all $K\leqslant H$,
\[
 \text{dim}|S^M|^K = |M^K|\leq |M^K| + n - 1 \leq \conn((A(S^{M\sqcup\bf{n}}))^K)
\]
by Lemma \ref{lem: connectivity} and hence $[|S^M|, |A(S^{M\sqcup \bf{n}})|]^H_* = 0$
by~\cite[Proposition 2.5]{Adams pre}.
\end{proof}
\end{subsubsection}

\begin{subsubsection}{Very special $\Gamma$-$G$-spaces and $G\Omega$-spectra}
\label{subsubsection: Very special Gamma G spaces and G Omega spectra}

We briefly recall how one obtains $G\Omega$-spectra from very special $\Gamma$-$G$-spaces
(cf.~\cite{Shimakawa}).
The next lemma and proposition are simplicial analogs of~\cite[Lemma 7.5, Theorem 7.6]{Santhanam}.

\begin{Lemma}
\label{lem: X otimes A very special}
 Suppose a $\Gamma$-$G$-space $A$ is special. Then, for any based $G$-simplicial set $X$, the $\Gamma$-$G$-space $A(X)$, defined by
 $n^+\mapsto A(n^+\wedge X)$ is special, too. If $A$ is very special, then $A(X)$ is very special, too.
\end{Lemma}
\begin{proof}
We first show that $A(X)$ is special again provided $A$ is special.
We need to show that the map $A(S^+\wedge X)\rightarrow \Map_\sset(S^+, A(X))$
is a $G$-equivalence for any based finite $G$-set $S^+$.
This morphism is the diagonal of a morphism of bisimplical $G$-sets, hence it suffices to
check that, for all $n\geq 0$, $A(S^+\wedge X_n)\rightarrow \Map_\sset(S^+, A(X_n))$
is a $G$-equivalence.
But both sides are $G$-equivalent to the weak product $\Map_\sset'(S^+\wedge X_n, A(1^+))$,
since $A$ is special.

Now, assume that $A$ is very special.
We have to show that $\pi_0(A(X))$ is a group.
Equivalently, we have to show that the map $A(X\vee X)\rightarrow A(X)\times A(X)$
induced by retraction onto the first summand and the fold map is a $G$-equivalence.
But again, it suffices to show that, for all $n\geq 0$, the map $A(X_n\vee X_n)\rightarrow A(X_n)\times A(X_n)$
is a $G$-equivalence. This is so because $A$ is very special.
\end{proof}

\begin{Prop}
\label{prop:  G Omega spectra from very special Gamma G spaces} 
 Suppose $A$ is a very special $\Gamma$-$G$-space. Then the $G$-symmetric spectrum $S(|A(\Sp)|)$, where
 $|-|$ and $S(-)$ denote geometric realization and singular complex functor respectively, is a $G\Omega$-spectrum.
\end{Prop}
\begin{proof}
Fix a subgroup $H\leqslant G$. Suppose $M$ is any finite $H$-set. In view of Lemma \ref{lem: cofibrant approximation} and Lemma \ref{lem: X otimes A very special},
the cofiber sequence
\[
 \xymatrix{
 S^T\wedge S^0\ar[r] &S^T\wedge \Delta^1 \ar[r] & S^{T\sqcup {\bf{1}}}
 }
\]
induces an $H$-fibration sequence upon applying $A(-)$ (cf.~\cite[Lemma 1.4. P2]{Shimakawa}).
This implies that the adjoint of the geometric realization of the structure map $\sigma_T^{\bf{1}}$ is an $H$-equivalence.
By~\cite[Theorem B]{Shimakawa}, for any two $H$-sets $M$ and $N$, the adjoint of the the geometric realization of the structure map $\sigma_{S\sqcup\bf{1}}^T$
is an $H$-equivalence, too.
It follows that in the diagram
\[
\xymatrix{
|A(S^{M})| \ar[r]^-{\simeq}\ar[d] & \Map_{\mathcal{T}_*}( S^1, |A(S^{M\sqcup {\bf{1}}})|)\ar[d]^{\simeq}\\
\Map_{\mathcal{T}_*}(|S^N|, |A(S^{M\sqcup N})|) \ar[r]^-{\simeq} & \Map_{\mathcal{T}_*}(|S^{N\sqcup {\bf{1}}}|, |A(S^{M\sqcup N\sqcup {\bf{1}}})|)
}
\]
the leftmost map is an $H$-equivalence, too.
Finally, since $S(|A(\Sp)|)$ is $G$-levelwise Kan, it is a $G\Omega$-spectrum.
\end{proof}
\end{subsubsection}
\end{subsection}

\begin{subsection}{Comparison of very special $\Gamma$-$G$-spaces and connective $G\Omega$-spectra}
 The last thing we need before comparing suitable homotopy categories of
 very special $\Gamma$-$G$-spaces and $G\Omega$-spectra
 is a special instance of the Wirthm\"uller isomorphism.
\begin{Lemma}
\label{lem: Wirthmueller}
 For any finite $G$-set $S$ the inclusion
\begin{equation*}
\entrymodifiers={+!! <0pt, \fontdimen22\textfont2>}
\xymatrix{
 \bigvee_{S} \Sp\ar[r] &\prod_{S} \Sp
}
\end{equation*}
is a $\pi_*$-isomorphism, hence a $G$-stable equivalence, of $G$-symmetric spectra.
\end{Lemma}
\begin{proof}
Both spectra are connective, so it suffices to show that all non-negative homotopy groups
of the cofiber vanish.
We choose a $G$-set $M$ which contains every orbit type at least once and which admits an injective map $\iota\colon S\rightarrow M$.
It suffices to show that
$[|S^{k\cdot M+ q}\wedge {G/H}_+|, |\prod_{S}S^{k\cdot  M}/\vee_{S}S^{k\cdot M}|]^G_* = 0$ for
all $H\leqslant G$, $q\in \N$  and $k\geq k_0$.
This follows from~\cite[Prop. 2.5]{Adams pre} if we can show that, for all $L$, eventually,
 \begin{equation}
 \label{eq: inequality dimension connectivity}
  k|M^L|+ q\leq \text{conn}\left(\prod_{S}S^{k\cdot M}/\bigvee_{S}S^{k\cdot M}\right)^L.
 \end{equation}
Now, $L$-equivariantly we have a decomposition $S^+ \cong_L L/J_1^+ \vee\ldots \vee L/J_n^+$ and then
\[
\left(\prod_{S}S^{k\cdot M}/\bigvee_{S}S^{k\cdot M}\right)^L \cong \prod_{i = 1}^n S^{k|M^{J_i}|}/\bigvee_{i\colon J_i = L}S^{k|M^{J_i}|}.
\]

There are two cases to distinguish.
If $J_i = L$ for all $i$, the connectivity of this space is at least $2k|M^L|-1$,
so that (\ref{eq: inequality dimension connectivity}) holds from some $k_0$ on.
Otherwise, there is at least one $i$ with $J_i \neq L$ and the connectivity is at least $k\cdot \text{min}_{i\colon J_i\neq L} |M^{J_i}|-1$.
But $|M^{J_i}|>|M^L|$ for all $i$ such that $J_i\neq L$, hence (\ref{eq: inequality dimension connectivity}) holds for $k$ large enough.
\end{proof}
We can now prove the equivariant analogon of~\cite[Theorem 5.1]{BF}.
\begin{Thm}
\label{thm: Comparison of very special Gamma G spaces and G Omega spectra}
The derived adjoint functors
\begin{equation*}
\xymatrix{
 \mathsf{Ho(\Gamma\text{-}G\text{-}spaces^{strict}}) \ar@<.3 ex>[r] &  \mathsf{Ho({GSp^\Sigma}^{flat\ level})} \ar@<.3 ex>[l]
}
\end{equation*}
restrict to mutually inverse equivalences of categories
when restricted to the full subcategories given by
very special $\Gamma$-$G$-spaces and $G$-symmetric spectra which
are $G$-level equivalent to connective $G\Omega$-spectra
respectively.
\end{Thm}

\begin{proof}
 Given a very special $\Gamma$-$G$-space, we have seen that $A(\Sp)$ is $G$-level equivalent to a
 $G\Omega$-spectrum.
 Suppose $X$ is a $G$-symmetric spectrum $X$ which is $G$-level equivalent to a connective $G\Omega$-spectrum.
 Then a fibrant replacement $X_f$ in the flat level model structure is a connective $G\Omega$-spectrum,
 since it is $G$-levelwise Kan by~\cite[p. 17]{Hausmann}.
 The $\Gamma$-$G$-space associated to $X_f$ is special by Lemma \ref{lem: Detection of weak equivalences},
 because for any finite $G$-set $S$ the inclusion
 $\vee_{S} \Sp\rightarrow \prod_{S} \Sp$
 is a $G$-stable equivalence by the Wirthm\"uller isomorphism and a $G$-flat cofibration between $G$-flat spectra
 (Proposition \ref{prop: product G flat}, Proposition \ref{prop: characterization flatness}).
 It is grouplike because $\pi_0^H(\Phi(\Sp, X)(1^+))$ 
 is a group and the monoid structures on
 $\pi_0^H(\Phi(\Sp, X)(1^+)) \cong \pi_0^H(X_0)$
 coincide.
 So the functors are well-defined.
 
 Suppose $A$ is very special and $X$ is a $G$-level fibrant $G\Omega$-spectrum. 
 If $A(\Sp)\rightarrow X$ is a $G$-level equivalence, then
 its adjoint $A\rightarrow \Phi(\Sp, X)$ is a level equivalence,
 since both are very special and
 \[
 A(1^+) \simeq_G X_0 \cong\Phi(\Sp, X)(1^+).
 \]
Conversely, suppose $A\rightarrow \Phi(\Sp, X)$ is a level equivalence. Firstly,
$A(\Sp)\rightarrow \Phi(\Sp, X)(\Sp)$ is a $G$-level equivalence
by Lemma \ref{lem: diag}.
Secondly, the map
$\Phi(\Sp, X)(\Sp)\rightarrow X$ is a $\pi_*$-isomorphism
because $\Phi(\Sp, X)(1^+) \cong X_0$.
And thirdly, a $\pi_*$-isomorphism of $G\Omega$-spectra is a $G$-level equivalence.
A proof of this statement in the setting of $G$-orthogonal spectra can be found in~\cite[Section 9]{MM}.
The arguments given there apply to our situation as well, because $G\Omega$-spectra
are by assumption $G$-levelwise Kan.
\end{proof}

\end{subsection}
\end{section}

\newpage
\begin{section}{Stable comparison}
\label{sec: Stable model structures}
We introduce stable model structures for the projective
and the strict model structures.
The existence of the localizations follows from general results
of localizations (cf.~\cite{Hirschhorn}) or can be proven along the lines of~\cite[Appendix A]{sthomalg}.
A map $f\colon A\rightarrow B$ of $\Gamma$-$G$-spaces is a \emph{stable equivalence}
if $f(\Sp)$ is a $\pi_*$-isomorphism of $G$-symmetric spectra.
A map between $\Gamma$-$G$-spaces
is called a \emph{stable} \emph{fibration} if it satisfies the right lifting
property with respect to
all projective cofibrations which are stable equivalences. We define a map to be a 
\emph{stable strict fibration} if it satisfies the right lifting property with respect
to all strict cofibrations which are in addition stable equivalences.

\begin{Rmk}
It follows from the discussion in~\cite[p. 65]{Hausmann}, that $f$ is a stable equivalence of $\Gamma$-$G$-spaces
if and only if $f(\Sp)$ is a $G$-stable equivalence of $G$-symmetric spectra.
\end{Rmk}

\begin{Thm}[Stable projective model structure]
\label{thm: stable projective model structure}
 The classes of projective cofibrations, stable fibrations and stable equivalences define a
 cofibrantly generated left proper $G$-simplicial model category structure
 on the category of $\Gamma$-$G$-spaces.
 The stably fibrant objects are precisely the very
 special $\Gamma$-spaces $X$ for which in addition
 $X(S^+)^H$ is Kan for all finite based $G$-sets $S^+$ and all subgroups $H\leqslant G$.
\end{Thm} 

\begin{Thm}[Stable strict model structure]
\label{thm: stable strict model structure}
The classes of strict cofibrations, stable strict fibrations and stable equivalences
define a cofibrantly generated $G$-simplicial model category structure on the category of $\Gamma$-$G$-spaces.
An object is stably strictly fibrant if and only if it is  strictly
fibrant and very special.
\end{Thm}
\begin{Cor}
The identity functor induces a Quillen equivalence between the stable
projective and the stable strict model structures.
\end{Cor}

Now we are in the position to prove the equivariant analogon of~\cite[Theorem 5.8]{BF}.
\begin{Thm}
\label{thm: stable comparison}
The derived adjoint functors
 \begin{equation*}
\xymatrix{
\mathsf{Ho(\Gamma\text{-}G\text{-}spaces^{stable\ strict/ stable\ projective})} \ar@<.3 ex>[r] &  \mathsf{Ho({GSp^\Sigma}^{flat\ stable})}. \ar@<.3 ex>[l] 
}
\end{equation*}
restrict to mutually inverse equivalences of categories when
the right adjoint is restricted to
the full subcategory given by connective $G$-symmetric spectra.
\end{Thm}
\begin{proof}
 Consider a strictly cofibrant $\Gamma$-$G$-space $A$ and a connective $G$-level fibrant $G\Omega$-spectrum $X$.
 Then $A\rightarrow \Phi(\Sp, X)$ is a stable equivalence if and only if $A(\Sp)\rightarrow X$ is a $\pi_*$-isomorphism
 (see the proof of Theorem \ref{thm: Comparison of very special Gamma G spaces and G Omega spectra}).
 This implies that unit and counit of this adjunction are isomorphisms.
\end{proof}

For later usage we put the following on record.
\begin{Lemma}
\label{lem: les htpy incl}
Fix a complete $G$-set universe $\mathcal{U}$.
If $i\colon A\rightarrow B$ is a map of $\Gamma$-$G$-spaces which is levelwise injective, then, for all subgroups $H\leqslant G$, there
is a long exact sequence
\begin{gather*}
\resizebox{12,5cm}{!}{
\xymatrix{
 \ldots \ar[r] & \pi_n^{H,\mathcal{U}}(A(\Sp))\ar[r] & \pi_n^{H,\mathcal{U}}(B(\Sp))\ar[r] & \pi_n^{H,\mathcal{U}}(B/A(\Sp))\ar[r] & \ldots \ar[r] & \pi_0^{H,\mathcal{U}}(B/A(\Sp))\ar[r] & 0.
}
}
\end{gather*}
\end{Lemma}
\begin{proof}
Indeed, being a colimit, taking the cone of a map commutes with the left adjoint
$(-)(\Sp)$.
Moreover, the map $C(i)\rightarrow B/A$ is a level equivalence of
$\Gamma$-$G$-spaces
and so the map $C(i)(\Sp)\rightarrow B/A(\Sp)$
is a $G$-level equivalence (cf. Lemma \ref{lem: diag}),
in particular it is a $\pi_*$-isomorphism.
The result thus follows from the usual long exact sequence of homotopy groups
(cf.~\cite[Proposition 2.10.11]{Hausmann})
and the connectivity of the spectra obtained from $\Gamma$-$G$-spaces.
\end{proof}

\end{section}

\newpage
\begin{section}{The smash product of \texorpdfstring{$\Gamma$-$G$-spaces}{Gamma-G-spaces}}
 \label{sec: The smash product of Gamma G spaces}
\begin{subsection}{The smash product of $\Gamma$-$G$-spaces}
\label{subsec: The smash product of Gamma G spaces}
In~\cite{Lydakis}, Lydakis defined a smash product for $\Gamma$-spaces.
To begin with, we choose a smash product functor
$\Gamma\times\Gamma\rightarrow \Gamma$
(for example by identifying the usual smash product $m^+\wedge n^+$ with $(mn)^+$
via $(i, j)\mapsto (i-1)n + j$).
Given two $\Gamma$-spaces $F$ and $F'$ the $n$th level of the smash product $F\wedge F'$ is given by
\[
\xymatrix{
 (F\wedge F')(n^+) = \colim_{k^+\wedge l^+\rightarrow n^+} F(k^+)\wedge F(l^+).
}
\]
$F\wedge F'$ is characterized by the property that maps of $\Gamma$-spaces $F\wedge F'\rightarrow T$ correspond
bijectively to maps of $\Gamma\times\Gamma$-spaces $F(-)\wedge F'(-)\rightarrow T(-\wedge -)$.
Elements in the smash product are represented by triples $[f, x\wedge y]$, where
$f\colon k^+\wedge l^+\rightarrow n^+$ is a morphism in $\Gamma$ and $x\wedge y\in F(k^+)\wedge F(l^+)$.
We define the internal mapping object to be the $\Gamma$-space $Hom(F, F')(m^+):= Map_{\Gamma(\sset)}(F, F'(m^+\wedge -))$.
Recall from~\cite[Theorem 2.18.]{Lydakis} that with these definitions, the category of $\Gamma$-spaces is closed symmetric monoidal
with unit $\Gamma(1^+, -)$.

Consequently, the category of $\Gamma$-$G$-spaces is a closed symmetric monoidal
category with unit $\Gamma(1^+, -)$ by defining the smash product
of two $\Gamma$-$G$-spaces to be the smash product of the underlying $\Gamma$-spaces
endowed with the diagonal $G$-action and equipping the internal mapping object with the conjugation action.
\end{subsection}
\begin{subsection}{Smash product and cofibrations}
\label{subsec: Smash product and cofibrations}

We study the pushout product of two strict cofibrations (resp.
projective cofibrations).

\begin{Lemma}
\label{lem: smash product of cofibrant again cofibrant}
 If $F$ and $F'$ are strictly cofibrant (resp. projectively cofibrant), then so is $F\wedge F'$.
\end{Lemma}
\begin{proof}
A $\Gamma$-$G$-space is strictly cofibrant if and only if
its underlying $\Gamma$-space is strictly cofibrant in the sense of Bousfield and Friedlander~\cite{BF}.
So in the case of strict cofibrations, this follows from the non-equivariant case~\cite[Lemma 4.5]{Lydakis}.

In the case of projective cofibrations, this follows from the fact that $Hom(F, -)$ preserves level fibrations
which are level equivalences if $F$ is projectively cofibrant because the projective model structure is $G$-simplicial.
\end{proof}

\begin{Prop}
\label{prop: pushout product of two cofibrations}
 If $F\rightarrow F'$, $\tilde{F}\rightarrow \tilde{F}'$ are two strict cofibrations
 (resp. projective cofibrations)
 of $\Gamma$-$G$-spaces, then the pushout product map
 \[
 \xymatrix{
  F\wedge \tilde{F}'\cup_{F\wedge \tilde{F}} F'\wedge \tilde{F}\ar[r]& F'\wedge \tilde{F}'
}
  \]
is a strict cofibration (resp. projective cofibration).
\end{Prop}
\begin{proof}
The pushout product map is injective by~\cite[Proposition 4.4]{Lydakis}
and has cofiber isomorphic to $F'/F\wedge \tilde{F}'/\tilde{F}$
which is strictly cofibrant (resp. projectively cofibrant) by the previous lemma.
This implies the claim, since a map is a strict cofibration (resp. projective cofibration) if and only if
it is injective and its cofiber is strictly cofibrant (resp. projectively cofibrant)
(In the case of projective cofibrations cf.~\cite[Lemma A3]{sthomalg}. The \emph{free} $\Gamma$-$G$-spaces are those of the form
$\bigvee_i G^+\wedge_{H_i} \Gamma_{H_i}(S_i^+, -)$ defined below.).
\end{proof}
\end{subsection}

\begin{subsection}{Smash product and level equivalences}
\label{subsec: Smash product and level equivalences}
We show that smashing with a strictly cofibrant $\Gamma$-$G$-space preserves
level equivalences.

\begin{Prop}
\label{prop: filtration gamma space}
 For any $\Gamma$-$G$-space $F$ and any positive integer $m$,
 there is a pushout square of $\Gamma$-$G$-spaces
 \[
 \xymatrix{
  \partial(\Gamma(m^+, -)\wedge_{\Sigma_m} F(m^+)) \ar[r]\ar[d] & \Gamma(m^+, -)\wedge_{\Sigma_m} F(m^+)\ar[d]\\
  (\sk_{m-1}F) \ar[r] & (\sk_{m}F),
}
  \]
where
{\tiny
\begin{gather*}
\partial(\Gamma(m^+, -)\wedge_{\Sigma_m} F(m^+))=\\
                                       \Gamma(m^+, -)\wedge_{\Sigma_m} (\sk_{m-1}F)(m^+)\cup_{(\sk_{m-1}\Gamma(m^+,-))\wedge_{\Sigma_m} (\sk_{m-1}F)(m^+)} (\sk_{m-1} \Gamma(m^+, -))\wedge_{\Sigma_m} F(m^+).
\end{gather*}
}
\end{Prop}
\begin{proof}
This follow from the nonequivariant case~\cite[Theorem 3.10]{Lydakis}, because the forgetful functor from $G$-spaces to spaces detects pushouts. 
\end{proof}

Next we prove an equivariant analogon of~\cite[Proposition 3.11]{Lydakis}.
Given a subgroup $H\leqslant G$ and a based finite $H$-set $S^+$
we will encounter $\Gamma$-$G$-spaces of the form $G^+\wedge_H \Gamma_H(S^+, -)$.
Here in the quotient we identify $gh\wedge \phi$ and $g\wedge \phi(h^{-1})$. 
The group $G$ acts on the left smash factor.

\begin{Prop}
\label{prop: discrete gamma set}
For any strictly cofibrant $\Gamma$-$G$-set $F$ and any $n\geq 0$
there exists a pushout diagram
\[
 \xymatrix{
\bigvee_{i\in I} G^+\wedge_{H_i}(\sk_{n-1}\Gamma_{H_i}(S_i^+, -))\ar[r]\ar[d] & \bigvee_{i\in I} G^+\wedge_{H_i}\Gamma_{H_i}(S_i^+, -)\ar[d]\\
 (\sk_{n-1}F) \ar[r] & (\sk_nF)
 }
\]
where $I$ is a set and for $i\in I$, $H_i\leqslant G$ is a subgroup and $S_i^+$ is based finite $H_i$-set.

\end{Prop}
\begin{proof}
There are elements $s_i\in F(n^+)$ and subgroups
$\Gamma_i = \{(h, \rho_i(h))|\ H_i\leqslant G,\ \rho_i\colon H_i\rightarrow \Sigma_n\ \text{homomorphism}\}$,
such that
\[
(F/(\sk_{n-1}F))(n^+) \cong \bigvee_{i\in I} (G\times \Sigma_n/\Gamma_i)\cdot s_i.
\]
For each $i$, we have the following isomorphism of $\Gamma$-$G$-spaces
\[
\resizebox{12,5cm}{!}{
\xymatrix{
 \Gamma(n^+, -)\wedge_{\Sigma_n} (G\times \Sigma_n/\Gamma_i)^+\ar[r] & G^+\wedge_{H_i} \Gamma_{H_i}(S_i^+, -),\ [f\wedge (g, \sigma)\Gamma_i]\mapsto [g\wedge (f\sigma)],
}
}
\]
where on the right $S_i^+$ denotes the set $\{1, \ldots, n\}$ equipped with the $H_i$-action coming from $\rho_i$.
So the $s_i$ give rise to a map
\[
\xymatrix{
\bigvee_{i\in I} G^+\wedge_{H_i}\Gamma_{H_i}(S_i^+, -)\ar[r] & \Gamma(n^+, -)\wedge_{\Sigma_n} F(n^+).
}
 \]

The image of this map intersects $\partial(\Gamma(n^+, -)\wedge_{\Sigma_n} F(n^+))$ precisely
in $\bigvee_{i\in I} G^+\wedge_{H_i}(\sk_{n-1}\Gamma_{H_i}(S_i^+, -))$
and any element can be lifted either to $\partial(\Gamma(n^+, -)\wedge_{\Sigma_n} F(n^+))$ or to $\bigvee_{i\in I} G^+\wedge_{H_i} \Gamma_{H_i}(S_i^+, -)$.
It follows that we have a pushout square
\[
\xymatrix{
\bigvee_{i\in I} G^+\wedge_{H_i}(\sk_{n-1}\Gamma_{H_i}(S_i^+, -))\ar[r]\ar[d] & \bigvee_{i\in I} G^+\wedge_{H_i}\Gamma_{H_i}(S_i^+, -)\ar[d]\\
\partial(\Gamma(n^+, -)\wedge_{\Sigma_n} F(n^+)) \ar[r] & (\Gamma(n^+, -)\wedge_{\Sigma_n} F(n^+)).
}
\]
Together with Proposition \ref{prop: filtration gamma space}, this proves the result.
\end{proof}

\begin{Prop}
\label{prop: smashing with strictly cofibrant preserves level equivalences}
 Smashing with a strictly cofibrant $\Gamma$-$G$-space
 preserves level equivalences.
\end{Prop}
\begin{proof}
Consider a strictly cofibrant $\Gamma$-$G$-space $F$ and a level equivalence $f\colon A\rightarrow B$.
The map $F\wedge A\rightarrow F\wedge B$ is the diagonal of the map bisimplicial $\Gamma$-$G$-sets
$(F_n\wedge A)_m\rightarrow (F_n\wedge B)_m$.
Here the subscript denotes the simplicial degree.
So we may assume that $F$ a strictly cofibrant $\Gamma$-$G$-set.
Suppose for a moment that, for any $H\leqslant G$ and any based $H$-set $S^+$, smashing with $G^+\wedge_H \Gamma_H(S^+, -)$ preserves level equivalences.
In view of Proposition \ref{prop: discrete gamma set} it follows inductively that $(\sk_n F)\wedge f$ is a level equivalence for all $n\geq0$.
Then $F\wedge f$ is a level equivalence, because homotopy groups commute with filtered colimits along $G$-cofibrations
and the maps $(\sk_n F)\wedge X\rightarrow (\sk_{n+1} F)\wedge X$ are strict cofibrations by Proposition \ref{prop: pushout product of two cofibrations}.

We now prove that smashing with $G^+\wedge_H \Gamma_H(S^+, -)$ preserves level equivalences.
Let $X$ be an arbitrary $\Gamma$-$G$-space and let $T^+$ be any based finite $G$-set. Then we have an isomorphism (natural in $T^+$)
\[
\xymatrix{
(G^+\wedge_H\Gamma_G(S^+, -)\wedge X)(T^+) \ar[r] & G^+\wedge_H X(\Gamma_H(S^+, T^+))
}
\]
given by mapping a tuple $[f, [g\wedge \phi]\wedge x]$ consisting of $f\colon k^+\wedge l^+\rightarrow T^+$, $\phi\colon S^+\rightarrow k^+$ and $x\in X(l^+)$ to
$[g\wedge X(\widetilde{f\circ (\phi\wedge l^+)})(x)]$, where $\widetilde{f\circ (\phi\wedge l^+)}: l^+ \rightarrow \Gamma_H(S^+, T^+)$ is the adjoint
of the composition $S^+\wedge l^+\rightarrow k^+\wedge l^+ \rightarrow T^+$ indicated.
$G$ acts from the left by
\[
 g'\cdot [g\wedge x] = [g'g\wedge X(\Gamma_H(S^+, g'))(x)],
\]
where $\Gamma_H(S^+, g')$ denotes postcomposition of maps with the action of $G$ on $T^+$.

Now let $K\leqslant G$ be a subgroup. We choose a set of representatives $\{g_i\}$ of $(G/H)^K = \{gH\colon K^g\leqslant H\}$.
Then
\[
(G^+\wedge_H X(\Gamma_H(S^+, T^+)))^K = \bigvee_{i} X(\Gamma_H(S^+, T^+))^{K^{g_i}}.
\]
This implies the claim.
\end{proof}
\end{subsection}

\begin{subsection}{The functor $(-)(\Sp)$ and smash products}
\label{subsec: The functor Sp otimes and smash products}
Recall that the category of $G$-symmetric spectra is a symmetric monoidal category
under the smash product with unit $\Sp$. The main result of this section is
Theorem \ref{thm: assembly} below, which states that $(-)(\Sp)$ takes smash
products to smash products up to $\pi_*$-isomorphism at least when one of the factors is strictly cofibrant.

The functor $(-)(\Sp)\colon \Gamma(G\sset)\rightarrow GSp^\Sigma$ is lax symmetric monoidal.
Indeed, given two $\Gamma$-$G$-spaces $F$ and $F'$, the natural maps
\[
\xymatrix{
 F(n^+)\wedge F'(m^+)\ar[r] & (F\wedge F')(n^+\wedge m^+)
}
 \]
induce a map
\[
\xymatrix{
 F(X)\wedge F'(Y)\ar[r] & (F\wedge F')(X\wedge Y),
}
\]
natural in based $G$-spaces $X$ and $Y$.
This in turn induces a bimorphism of spectra from the pair $(F(\Sp), F'(\Sp))$
to $(F\wedge F')(\Sp)$ which gives rise to the
natural transformation
\[
 \xymatrix{
 a_{F, F'}\colon F(\Sp)\wedge F'(\Sp)\ar[r] & (F\wedge F')(\Sp).
 }
\]
Moreover, sending $x\in S^M$ to $[x, \id_{1^+}]$ induces an isomorphism $\lambda\colon \Sp\rightarrow \Gamma(1^+, -)(\Sp)$.
Now several coherence diagrams have to be checked, which we skip (cf.~\cite[Proposition 3.3 and p. 442]{MMSS}).
\begin{Thm}
\label{thm: assembly}
 The map $a_{X, Y}$ is a $\pi_*$-isomorphism, in particular a $G$-stable
 equivalence, if $X$ or $Y$ is strictly cofibrant.
\end{Thm}
\begin{proof}
In view of Propositions \ref{prop: Smashing with flat G symmetric spectrum preserves level equivalences and pi * isomorphisms} and 
\ref{prop: smashing with strictly cofibrant preserves level equivalences} and Lemma \ref{lem: diag} we may assume that $X$ and $Y$ are projectively cofibrant.
If we fix $Y$, then the class of $\Gamma$-$G$-spaces $X$ for which the assembly map is a $\pi_*$-isomorphism
is closed under pushouts along generating projective cofibrations,
filtered colimits along projective cofibrations
and retracts.
This reduces to consider $X = \Gamma_G(S_1^+, -)$ for some based finite $G$-set $S^+$
and applying the same reasoning again reduces to $Y = \Gamma_G(S_2^+, -)$ for some based finite $G$-set $S_2^+$.
In this case we have to show that
\[
\xymatrix{
\Sp^{\times S_1}\wedge \Sp^{\times S_2}\ar[r]& \Sp^{\times (S_1\times S_2)}
}
\]
induced by the bimorphism
\[
\xymatrix{
(S^{\bf{n}})^{\times S_1}\wedge (S^{\bf{m}})^{\times S_2}\ar[r] & (S^{\bf{n\sqcup m}})^{\times (S_1\times S_2)}, ((x_i), (y_j))\mapsto (x_i\wedge y_j)
}
\]
is a $\pi_*$-isomorphism.
Precomposition with the $\pi_*$-isomorphism (Proposition \ref{prop: Smashing with flat G symmetric spectrum preserves level equivalences and pi * isomorphisms}, Proposition \ref{prop: characterization flatness}
and Lemma \ref{lem: Wirthmueller}) 
$\Sp^{\vee S_1}\wedge \Sp^{\vee S_2}\rightarrow \Sp^{\times S_1}\wedge \Sp^{\vee S_2}\rightarrow \Sp^{\times S_1}\wedge \Sp^{\times S_2}$
is a $\pi_*$-isomorphism by Lemma \ref{lem: Wirthmueller}.
Hence the map is a is a $\pi_*$-isomorphism as well.
\end{proof}

\begin{Prop}
\label{prop: pushout product and monoid axiom}
\begin{itemize}
\item[(a)] Smashing with a strictly cofibrant $\Gamma$-$G$-space
 preserves stable equivalences.
 
\item[(b)]\text{(Pushout product axiom)} If $F\rightarrow F'$, $\tilde{F}\rightarrow \tilde{F}'$ are two strict cofibrations
 (resp. projective cofibrations)
 of $\Gamma$-$G$-spaces, then the pushout product map
 \[
 \xymatrix{
  F\wedge \tilde{F}'\cup_{F\wedge \tilde{F}} F'\wedge \tilde{F}\ar[r]& F'\wedge \tilde{F}'
}
  \]
is a strict cofibration (resp. projective cofibration).
If in addition one of the former maps is a stable equivalence, then
so is the pushout product.

\item[(c)]\text{(Monoid Axiom)} Let $I$ denote the smallest class of maps of $\Gamma$-$G$-spaces which contains
the maps of the form $A\wedge Z\rightarrow B\wedge Z$, where $A\rightarrow B$ is a
stable equivalence and a projective cofibration (resp. strict cofibration)
and which is closed under cobase change and transfinite composition.
Then every map in $I$ is a stable equivalence.
\end{itemize}
\end{Prop}
\begin{proof}
The first part follows from Theorem \ref{thm: assembly} and Proposition \ref{prop: Smashing with flat G symmetric spectrum preserves level equivalences and pi * isomorphisms}.
The second part follows from Proposition \ref{prop: pushout product of two cofibrations}, Lemma \ref{lem: les htpy incl} and
the first part.
It remains to prove the third part. This is in analogy with~\cite[Lemma 1.7]{sthomalg}.
\end{proof}

\begin{Rmk}
Define a (commutative) $\Gamma$-$G$-ring to be a (commutative) monoid in the symmetric monoidal category $\Gamma(G\sset)$.
A left $R$-module is a $\Gamma$-$G$-space $M$ together with a map $R\wedge M\rightarrow M$
satisfying associativity and unit conditions.
Defining weak equivalences (resp. fibrations) to be stable equivalences (resp. stable fibrations or stable strict fibrations)
and cofibrations by the adequate lifting property,
it follows essentially from the previous proposition (cf.~\cite[Theorem 2.2]{sthomalg}) that, for any $\Gamma$-$G$-ring $R$,
the category of left $R$-modules becomes a cofibrantly generated closed $G$-simplicial model category.

Suppose $k$ is a commutative $\Gamma$-$G$-ring.
The category of left $k$-modules is a symmetric monoidal category with respect to the smash product
$A\wedge_k B$ which is the coequalizer of the two actions $A\wedge k\wedge B\rightrightarrows A\wedge B$
given by multiplication.

A $k$-algebra is then a monoid in $k$-modules and the category of $k$-algebras is a closed $G$-simplicial
model category when defining a map to be a weak equivalence (resp. fibration) if the underlying map of $k$-modules
has this property (cf.~\cite[Theorem 2.5]{sthomalg}).
\end{Rmk}
\end{subsection}
\end{section}

\newpage
\begin{section}{Geometric fixed points of \texorpdfstring{$\Gamma$-$G$-spaces}{Gamma-G-spaces}}
\label{sec: Geometric fixed points of Gamma G spaces}
In this section we construct a geometric fixed points functor
\[
\xymatrix{
\Phi^G\colon \Gamma(G\sset)\ar[r]& \Gamma(\sset).
}
\]
Given a $\Gamma$-$G$-space $A$, $\Phi^G A$
is defined to be the $\Gamma$-space given by $(\Phi^GA)(k^+) = A((k^+)^{\wedge G})^G$.
This is in fact a lax symmetric monoidal functor.
The transformation $(\Phi^GX)\wedge (\Phi^GY)\rightarrow \Phi^G(X\wedge Y)$ is induced by the map
\[
\xymatrix{
 X((k^+)^{\wedge G})^G\wedge Y((l^+)^{\wedge G})^G\ar[r] & (X\wedge Y)((kl^+)^{\wedge G})^G,\ (x\wedge y)\mapsto [\id, x\wedge y]
}
 \]
and the map $\Gamma(1^+, -)\rightarrow \Phi^G\Gamma(1^+, -)$ is defined to be the isomorphism
\[
\Gamma(1^+, k^+)\cong \Gamma(1^+, ((k^+)^{\wedge G})^G)\cong \Gamma(1^+, (k^+)^{\wedge G})^G.
\]

The functor $\Phi^G(-)$ enjoys several good properties, which we collect in the next two propositions.

\begin{Prop}
 A map $f\colon A\rightarrow B$ of $\Gamma$-$G$-spaces is a stable equivalence if and only if,
 for all $H\leqslant G$, the map $\Phi^H(f)\colon \Phi^HA\rightarrow \Phi^HB$ is a stable equivalence.
\end{Prop}
\begin{proof}
Given a $\Gamma$-$G$-space the $G$-symmetric spectrum (of spaces) $|A(\Sp)|$
is the underlying $G$-symmetric spectrum of a $G$-orthogonal spectrum $|A|(\Sp)$
(by abuse of notation, we denote the topological sphere spectrum by $\Sp$, too).
Moreover, $|(\Phi^HA)(\Sp)|$ is naturally isomorphic to the geometric fixed point spectrum $\Phi^H(|A|(\Sp))$
of the $G$-orthogonal spectrum  $|A|(\Sp)$~\cite{Schwede Lectures on equivariant stable homotopy theory}.
This follows from the fact that $|(S^{\bf{n}})^{\wedge H}|$ is isomorphic to the one point compactification $S^{n\rho_H}$
of $n$ copies of the regular representation $\rho_H$ of $H$.
Now, a morphism $f\colon A\rightarrow B$ of $\Gamma$-$G$-spaces is a $G$-stable equivalence
if and only if $A(\Sp)\rightarrow B(\Sp)$ is a $\pi_*$-isomorphism of $G$-symmetric spectra by definition.
This is the case if and only if $|A|(\Sp)\rightarrow |B|(\Sp)$ is a $\pi_*$-isomorphism
of $G$-orthogonal spectra~\cite[p. 65]{Hausmann}. Equivalently,
$\Phi^H(|A|(\Sp))\rightarrow \Phi^H(|B|(\Sp))$ is a $\pi_*$-isomorphism of orthogonal spectra for all subgroups $H\leqslant G$~\cite[Theorem 7.12]{Schwede Lectures on equivariant stable homotopy theory}.
And this is the case if and only if $\Phi^HA\rightarrow \Phi^HB$ is a stable equivalence
of $\Gamma$-spaces for all $H\leqslant G$~\cite[p. 65]{Hausmann}.
\end{proof}
\begin{Prop}
\begin{itemize}
\item[(a)] For any based finite $G$-set $S^+$, the map
\[
\xymatrix{
S^+\wedge \Gamma(1^+, -)\ar[r]& \Gamma_G(S^+,-),\ s\wedge \phi\mapsto (\phi\circ p_s)
}
\]
induces a stable equivalence $(S^+)^G\wedge \Gamma(1^+, -)\simeq \Phi^G(\Gamma_G(S^+, -))$.
\item[(b)] $(\Phi^GA)_{c}\wedge (\Phi^GB)\rightarrow \Phi^G(A\wedge B)$ is a stable equivalence whenever $A$ or $B$ is strictly cofibrant.
Here, $X_{c}$ denotes a cofibrant replacement in the stable strict model structure.
\end{itemize}
\end{Prop}
\begin{proof}
Part $(a)$ follows from the Wirthm\"uller isomorphism Lemma \ref{lem: Wirthmueller} and the previous proposition.
This implies that $(b)$ holds for $A = \Gamma_G(S_1^+, -)$ and $B = \Gamma_G(S_2^+, -)$.
If we fix this $B$, then the class of $\Gamma$-$G$-space for which $(b)$ holds is closed under pushouts along generating projective cofibrations
($\Phi^G(-)$ takes pushouts along cofibrations to pushouts), filtered colimits along projective cofibrations (since $\Phi^G(-)$ commutes with such colimits) and retracts.
Thus $A$ may be an arbitrary projectively cofibrant $\Gamma$-$G$-space and the same argument shows that $B$ can be an arbitrary projectively cofibrant $\Gamma$-$G$-space.
This finishes the proof in view of Proposition \ref{prop: smashing with strictly cofibrant preserves level equivalences}.
\end{proof}
\end{section}

\newpage
\begin{section}{Appendix}
The following sections contain several proofs deferred from other sections.

\begin{subsection}{The strict model structure for $\Gamma$-$G$-spaces}
\label{subsec: The strict model structure for Gamma G spaces}
The aim of this subsection is to prove Theorem \ref{thm: gen strict model structure}
below.
We start by observing that we have the following adjunction
for a based right $\Sigma_n$- and left $\Sigma_l\times G$-space A:
\[
\xymatrix@R=0cm{
 A\wedge_{\Sigma_n}-: (G\times \Sigma_n)\sset \ar@<.3 ex>[r] &  (G\times \Sigma_l)\sset: \Map_{\sset}(A,-)^{\Sigma_l}. \ar@<.3 ex>[l]
}
\]
Here, $G\times\Sigma_n$ acts on $\Map_{\sset}(A, -)^{\Sigma_l}$ by $((\sigma, g)\cdot f)(a) := gf(g^{-1}a\sigma)$.
For pointed sets $S^+$ and $T^+$, $\Inj_*(S^+, T^+)$ (resp. $\Surj_*(S^+, T^+)$)
denotes the set of based injective (resp. surjective) maps $S^+\rightarrow T^+$.
We make the following assumptions.
\begin{ass}
\label{assumptions}
\begin{itemize}
 \item[(a)] There are structures of model categories on $G\times\Sigma_n$-spaces
            denoted by $\mathcal{G}^1_n\sset$
       and $\mathcal{G}^2_n\sset$ respectively, such that 
       the first one is $G\times\Sigma_n$-simplicial.

\item[(b)]  The class of $\mathcal{G}^2_n$-equivalences is included in the class of
            $\mathcal{G}^1_n$-equivalences for all $n\geq 0$.
 
\item[(c)] The adjoint pairs
\[
\xymatrix@R=0cm{
 \Inj_*(l^+, n^+)^+\wedge_{\Sigma_l}-: \mathcal{G}^1_l\sset \ar@<.3 ex>[r] &  \mathcal{G}^2_n\sset: \Map_{\sset}(\Inj_*(l^+, n^+)^+, - )^{\Sigma_n}, \ar@<.3 ex>[l]\\
 \Surj_*(n^+, l^+)^+\wedge_{\Sigma_n}-: \mathcal{G}^1_n\sset \ar@<.3 ex>[r] &  \mathcal{G}^1_l\sset: \Map_{\sset}(\Surj_*(n^+, l^+)^+, - )^{\Sigma_l} \ar@<.3 ex>[l]
}
\]
are Quillen adjunctions.
\end{itemize}
\end{ass}

Let $\Gamma_{\leq n}$ denote the full subcategory of $\Gamma$ with objects the
sets $l^+$, $l\leq n$.
As in~\cite{BF}, the truncation functor
$T_n\colon \Gamma(G\sset)\rightarrow \Gamma_{\leq n}(G\sset)$
has both a left and a right
adjoint denoted by $\sk_n$ and $\csk_n$ respectively.
By abuse of notation, we will usually write $\sk_n$ (resp. $\csk_n$) for the composition
$\sk_n\circ T_n$ (resp. $\csk_n$), too.

Consider a map $f\colon X\rightarrow Y$ between $\Gamma$-$G$-spaces. Then $f$ is a \emph{strict cofibration}
if, for all $n\geq 0$, the map 
\[
\xymatrix{
i_n(f)\colon (\sk_{n-1}Y)(n^+)\cup_{(\sk_{n-1} X)(n^+)} X(n^+)\ar[r] & Y(n^+)
}
\]
is a $\mathcal{G}^1_n$-cofibration.
Dually, $f$ is a \emph{strict fibration} if, for all $n\geq 0$, the map
\[
\xymatrix{
p_n(f)\colon X(n^+)\ar[r] & (\csk_{n-1}X)(n^+)\times_{(\csk_{n-1}Y)(n^+)}Y(n^+)
}
\]
is a $\mathcal{G}^1_n$-fibration.
Finally, $f$ is a \emph{strict weak equivalence} if it
is levelwise a $\mathcal{G}^1_n$-equivalence.

We prove
\begin{Thm}
\label{thm: gen strict model structure}
Under these assumptions, the strict notions of weak equivalences, fibrations and cofibrations make the
category $\Gamma(G\sset)$ into a $G$-simplicial model category.
\end{Thm}
\begin{ex}
\begin{itemize}
\item Suppose $G$ is the trivial group.
We may take $\mathcal{G}_n^1\sset$ to be the model structure on
$\Sigma_n$-spaces where weak equivalences and fibrations are defined
by the forgetful functor to spaces and $\mathcal{G}_n^2$ to be the usual
model structure on $\Sigma_n$-spaces where weak equivalences and fibrations
are detected on all fixed points.
This recovers the model structure by Bousfield and Friedlander (cf.~\cite{BF}).

\item More generally, taking $\mathcal{G}_n^1\sset$ (resp. $\mathcal{G}_n^2\sset$)
to be the model structure with respect
to the family of subgroups of $G\times \Sigma_n$ that intersect $\{1\}\times \Sigma_n$ trivially
(resp. the family of all subgroups of $G\times \Sigma_n$) yields the
model structure applied throughout this paper.
\end{itemize}
\end{ex}

Before proving the theorem, we need a few preparations.

\begin{Prop}
\label{prop: ext}
Suppose $B$, $X\in \Gamma_{\leq n}(G\sset))$ and $u_{n-1}\colon T_{n-1}B \rightarrow T_{n-1}X$ is a map
in $\Gamma_{\leq n-1}(G\sset)$. A map $u^n\colon B(n^+)\rightarrow X(n^+)$ in $G\sset$ determines
a prolongation of $u_{n-1}$ to $u\colon B\rightarrow X$ in $\Gamma_{\leq n}(G\sset)$ if and only if
$u^n$ is $G\times \Sigma_n$-equivariant and fills in the following commutative diagram in $G\times\Sigma_n\sset$:
\begin{gather*}
 \xymatrix{
(\sk_{n-1} B)(n^+) \ar[r] \ar[d] & B(n^+) \ar[r]\ar[d] & (\csk_{n-1} B)(n^+) \ar[d]\\
(\sk_{n-1} X)(n^+) \ar[r] & X(n^+) \ar[r] & (\csk_{n-1} X)(n^+).
}
\end{gather*}
\end{Prop}
\begin{proof}
See Proposition 3.4 from~\cite{BF}.
\end{proof}

\begin{Prop}
 \label{prop: ext square}
Consider a diagram
\begin{equation}
\begin{gathered}
\label{dd}
 \xymatrix{
A \ar[r]\ar[d] & X\ar[d]\\
B \ar[r]       & Y
}
\end{gathered}
\end{equation}
in $\Gamma_{\leq n}(G\sset)$ and a map $T_{n-1}B\rightarrow T_{n-1}X$ which makes the diagram
\begin{gather*}
 \xymatrix{
T_{n-1}A \ar[r]\ar[d] & T_{n-1}X\ar[d]\\
T_{n-1}B \ar[r]\ar[ru]& T_{n-1}Y
}
\end{gather*}
commute. Then, the diagram \eqref{dd} has a lift $B\rightarrow X$ if there is a lift in
the diagram of $G\times\Sigma_n$-spaces
\begin{gather*}
 \xymatrix{
(\sk_{n-1}B)(n^+)\cup_{(\sk_{n-1}A)(n^+)} A(n^+) \ar[r]\ar[d] & X(n^+) \ar[d]\\
B(n^+)                                       \ar[r]       & (\csk_{n-1}X)(n^+)\times_{(\csk_{n-1}Y)(n^+)} Y(n^+).
}
\end{gather*}
\end{Prop}
\begin{proof}
This is a direct consequence of the preceding proposition.
\end{proof}

\begin{Prop}
\label{prop: diagram skeleta}
For any $\Gamma$-$G$-space $X$ and any positive integers $m$, $n\geq 0$, there is a pushout square
of $G\times \Sigma_n$-spaces
\begin{equation}
\begin{gathered}
\label{diagram skeleta}
\xymatrix{
 \Inj_*(l^+, n^+)^+\wedge_{\Sigma_l} (\sk_{l-1} X)(l^+) \ar[r] \ar[d] & (\sk_{l-1} X)(n^+) \ar[d]\\
 \Inj_*(l^+, n^+)^+\wedge_{\Sigma_l} X(l^+) \ar[r] & (\sk_l X)(n^+).
}
\end{gathered}
\end{equation}
Here the top and bottom horizontal maps are given by pushing forward along an element of $\Inj_*(l^+, n^+)$,
where one uses the canonical isomorphism $(\sk_l X)(l^+)\rightarrow X(l^+)$ for the lower one,
and the left and right vertical maps are induced by the canonical maps
$(\sk_{l-1}X)(l^+)\rightarrow X(l^+)$ and $(\sk_{l-1} X)(n^+)\rightarrow (\sk_l X)(n^+)$.
\end{Prop}
\begin{proof}
The diagram is a commutative diagram of $G\times \Sigma_n$-spaces
and its underlying diagram of spaces is isomorphic to
\[
\xymatrix{
 \binom{l}{n}^+\wedge (\sk_{l-1} X)(l^+) \ar[r] \ar[d] & (\sk_{l-1} X)(n^+) \ar[d]\\
 \binom{l}{n}^+\wedge  X(l^+) \ar[r] & (\sk_l X)(n^+),
}
\]
where $\binom{l}{n}$ denotes the set of order-preserving injections of the set $\{1, \ldots, l\}$ into the set $\{1, \ldots, n\}$
both endowed with the natural ordering.
Lydakis shows (cf.~\cite[Proposition 3.8]{Lydakis}) that for any $\Gamma$-space $X$ and any positive integers $m$, $n\geq 0$ this is a pushout diagram,
hence (\ref{diagram skeleta}) it is a pushout diagram in $G\times \Sigma_n$-spaces.
\end{proof}

\begin{Lemma}{\cite[Lemma 3.7]{BF}}
 \label{lem: sk cof}
 Let $n$ be a non-negative integer and fix $N\leq n$. Consider a map $f\colon A \rightarrow B$ in $\Gamma(G\sset)$.
 If the maps $i_m(f)$ are $\mathcal{G}^1_m$-cofibrations (resp. acyclic $\mathcal{G}^1_m$-cofibrations) for all $m\leq N$,
 then the maps $(\sk_{l} A)(n^+) \rightarrow (\sk_l B)(n^+)$ are $\mathcal{G}^2_n$-cofibrations (resp. acyclic $\mathcal{G}^2_n$-cofibrations)
 for all $l\leq N$.
\end{Lemma}
\begin{proof}
The case $l = 0$ is trivial. Assume inductively that the assertion holds true for all $l-1\leq N-1$.
By the first part of assumption $(c)$, $\Inj_*(l^+, n^+)^+\wedge_{\Sigma_l} i_l(f)$ is a $\mathcal{G}^2_n$-cofibration (resp. acyclic $\mathcal{G}^2_n$-cofibration).
In view of Proposition \ref{prop: diagram skeleta}, the inductive step can now be finished by applying Reedy's patching Lemma to the diagram
\[
\resizebox{12,5cm}{!}{
\xymatrix{
\Inj_*(l^+, n^+)^+\wedge_{\Sigma_l} A(l^+)\ar[d] & \Inj_*(l^+, n^+)^+\wedge_{\Sigma_l} (\sk_{l-1} A)(l^+) \ar[l]\ar[r] \ar[d] & (\sk_{l-1} A)(n^+)\ar[d]\\
\Inj_*(l^+, n^+)^+\wedge_{\Sigma_l} B(l^+) & \Inj_*(l^+, n^+)^+\wedge_{\Sigma_l} (\sk_{l-1} B)(l^+) \ar[l]\ar[r] & (\sk_{l-1} B)(n^+).
}
}
\]
\end{proof}

\begin{Lemma}
 \label{lem: ac str cof}
If $f\colon A\rightarrow B$ is an acyclic strict cofibration, then the maps
\[
\xymatrix{
 i_n(f)\colon A(n^+)\cup_{\sk_{n-1}A(n^+)} (\sk_{n-1}B)(n^+) \ar[r] & B(n^+)
}
\]
are in fact acyclic $\mathcal{G}^1_n$-cofibrations.
\end{Lemma}
\begin{proof}
 The case $n = 0$ is trivial.
 Assume inductively that $i_m(f)$ is an acyclic $\mathcal{G}^1_m$-cofibration
 for all $m \leq n-1$. We show that $i_{n}(f)$ is a $\mathcal{G}^1_n$-equivalence.
 To this end it suffices to show that
 $(\sk_{n-1}A)(n^+)\rightarrow (\sk_{n-1}B)(n^+)$ is an acyclic
 $\mathcal{G}^2_n$-cofibration,
 because this implies
 that $A(n^+)\rightarrow A(n^+)\cup_{\sk_{n-1}(A(n^+)} (\sk_{n-1}B)(n^+)$ is an
 acyclic $\mathcal{G}^2_n$-cofibration
 and, since $\mathcal{G}^2_n$-equivalences are in particular
 $\mathcal{G}^1_n$-equivalences by assumption $(b)$, the assertion follows then from two out of three
 for weak equivalences.
 But the map in question is an acyclic $\mathcal{G}^2_n$-cofibration by Lemma \ref{lem: sk cof} applied to the case $N = n-1$.
\end{proof}

There are dual results for fibrations.

\begin{Prop}
For any $\Gamma$-$G$-space $X$ and any positive integers $m$, $n\geq 0$, there is a pullback square
of $G\times \Sigma_n$-spaces
\begin{equation}
\begin{gathered}
\label{diagram coskeleta}
\entrymodifiers={+!! <0pt, \fontdimen22\textfont2>}
\xymatrix{
 (\csk_l X)(n^+) \ar[r] \ar[d] & \Map_{\sset}(\Surj_*(n^+, l^+)^+, X(l^+))^{\Sigma_l} \ar[d]\\
 (\csk_{l-1}X)(n^+) \ar[r] & \Map_{\sset}(\Surj_*(n^+, l^+)^+, (\csk_{l-1}X)(l^+))^{\Sigma_l}.
}
\end{gathered}
\end{equation}
Here the top and bottom horizontal maps are given by pushing forward along an element of $\Surj_*(l^+, n^+)$,
where one uses the canonical identification $(\csk_l X)(l^+)\rightarrow X(l^+)$ for the top map,
and the left and right vertical maps are induced by the canonical maps
$X(l^+)\rightarrow (\csk_{l-1}X)(l^+)$ and $(\csk_l X)(n^+)\rightarrow (\csk_{l-1} X)(n^+)$ respectively.
\end{Prop}
\begin{proof}
The proof is analogous to the proof of~\cite[Proposition 3.8]{Lydakis}.
\end{proof}

\begin{Lemma}{\cite[Lemma 3.7]{BF}}
\label{lem: csk fib}
 Consider a map of $\Gamma$-$G$-spaces $f\colon A \rightarrow B$ such that the maps $p_m(f)$ are $\mathcal{G}^1_m$-fibrations (resp. acyclic $\mathcal{G}^1_m$-fibrations) for all $m\leq N$,
 where $N\leq n$ is fixed.
 Then the maps $(\csk_{l} A)(n^+) \rightarrow (\csk_l B)(n^+)$ are $\mathcal{G}^1_n$-fibrations (resp. acyclic $\mathcal{G}^1_n$-fibrations) for all $l\leq N$.
\end{Lemma}
\begin{proof}
The case $l = 0$ is trivial. Assume inductively that the assertion holds true for $l-1\leq N-1$.
By Reedy's patching lemma it suffices to know that
$\Map_{\sset}(\Surj_*(n^+,l^+)^+, p_l)^{\Sigma_l}$ is a $\mathcal{G}^1_n$-fibration (resp. acyclic $\mathcal{G}^1_n$-fibration)
which follows by the second part of assumption $(c)$.
\end{proof}

\begin{Lemma}
 \label{lem: ac str fib}
 If $f\colon A\rightarrow B$ is an acyclic strict fibration, then the maps
 \[
 \xymatrix{
 p_n(f)\colon A(n^+)\ar[r] & (\csk_{n-1}A)(n^+)\times_{(\csk_{n-1}B)(n^+)}B(n^+)
}
\]
 are in fact acyclic $\mathcal{G}^1_n$-fibrations.
\end{Lemma}
\begin{proof}
 The case $n = 0$ is trivial.
 Assume inductively that $p_m(f)$ are acyclic $\mathcal{G}^1_m$-fibrations
 for $m \leq n-1$.
 We show that $p_{n}(f)$ is an acyclic $\mathcal{G}^1_n$-fibration.
 By the previous lemma in the case $N = n-1$, we have that
 $(\csk_{n-1}A)(n^+)\rightarrow (\csk_{n-1}B)(n^+)$
 is an acyclic $\mathcal{G}_n^1$-fibration.
Hence $(\csk_{n-1}A)(n^+)\times_{(\csk_{n-1}B)(n^+)}B(n^+)\rightarrow B(n^+)$
is an acyclic $\mathcal{G}_n^1$-fibration as well.
\end{proof}

\begin{proof}[Proof of Theorem \ref{thm: gen strict model structure}]
{\bf MC 1}, {\bf MC 2} and {\bf MC 3} are clear.
{\bf MC 4} follows immediately from Lemma \ref{lem: ac str cof}, Lemma \ref{lem: ac str fib} and
Proposition \ref{prop: ext square}. So we only have to show {\bf MC 5}, the existence of factorizations.
Given a map $f\colon A\rightarrow B$ in $\Gamma(G\sset)$,
assume inductively that it has already been factored up to level $n-1$ as an acyclic strict cofibration followed
by a strict fibration (resp. strict cofibration followed by an acyclic strict fibration) $T_{n-1}A \rightarrow C_{\leq n-1} \rightarrow T_{n-1}B$.
Then, as in~\cite{BF}, we obtain a diagram
\begin{equation}
\label{eq: faktorisierung strikt}
\begin{gathered}
 \xymatrix{
(\sk_{n-1} A)(n^+) \ar[r] \ar[d] &A(n^+) \ar[r]\ar[d] &(\csk_{n-1} A)(n^+) \ar[d]\\
(\sk_{n-1} C_{\leq n-1})(n^+) \ar[r] \ar[d] &K \ar[r]\ar[d] &(\csk_{n-1} C_{\leq n-1})(n^+) \ar[d]\\
(\sk_{n-1} B)(n^+) \ar[r]        &B(n^+) \ar[r] &(\csk_{n-1} B)(n^+),
}
\end{gathered}
\end{equation}
where $K$ comes from a factorization
\[
\resizebox{12,5cm}{!}{
\xymatrix{
(\sk_{n-1}C_{\leq n-1})(n^+)\cup_{(\sk_{n-1}A)(n^+)}A(n^+)\ar[r] & K\ar[r] & (\csk_{n-1}C_{\leq n-1})(n^+)\times_{(\csk_{n-1} B)(n^+)} B(n^+)
}
}
\]
of the canonical map into an acyclic cofibration followed by a fibration (resp. cofibration followed by an acyclic fibration)
in $\mathcal{G}^1_n\sset$.
The $G\times\Sigma_n$-space $K$ gives rise to an object $C_{\leq n}\in\Gamma_{\leq n}(G\sset)$
with $C_{\leq n}(k^+) = C_{\leq n-1}(k^+)$ for all $k\leq n-1$ and $C_{\leq n}(n^+) = K$,
such that the canonical factorization
\[
\xymatrix{ 
(\sk_{n-1} C_{\leq n})(n^+) \ar[r] & K \ar[r] & (\csk_{n-1} C_{\leq n})(n^+)
}
\]
equals the factorization in (\ref{eq: faktorisierung strikt}).

In any case, this produces a factorization $A\rightarrow C\rightarrow B$ as a strict cofibration followed by a strict fibration.
Assume that $K$ was always obtained by a factorization as an acyclic cofibration followed by a fibration.
We show that $A\rightarrow C$ is acyclic.
But this follows from Lemma \ref{lem: sk cof} for $l = n$.
In the other case, $K$ was always obtained by a factorization as cofibration followed by an acyclic fibration.
Then $C\rightarrow B$ is acyclic by Lemma \ref{lem: csk fib} in the case $l = n$.

Finally, this model structure is $G$-simplicial since, for all $n\geq 0$,
$i_n(j\square i)$ is isomorphic to $j\square i_n(f)$
for any strict cofibration $f$ and any $G$-cofibration $j$ of $G$-spaces.
\end{proof}
\end{subsection}

\begin{subsection}{A characterization of flat cofibrations}
\label{subsec: A characterization of flat cofibrations}
The aim of this section ist
to prove
\begin{Prop}
\label{prop: product G flat}
For any $G$-flat $G$-symmetric spectrum $X$ and any finite $G$-set $S$ of cardinality $n$,
the spectrum $X^{\times S}$ is $G$-flat and
the inclusion
\[
\xymatrix{
X^{\times S}_{\leq n-1} \ar[r] & X^{\times S}
}
\]
is a $G$-flat cofibration of $G$-symmetric spectra.
Here $X^{\times S}_{\leq n-1}$ is the subspectrum which is levelwise given by those tuples in the product
such that either two entries coincide or one of them equals the basepoint.
In particular, the
$G$-symmetric spectrum $\Sp^{\times S}$ is $G$-flat and the
inclusion
\[
\xymatrix{
 \Sp^{\times S}_{\leq n-1}\ar[r] & \Sp^{\times S}
}
 \]
is a $G$-flat cofibration of $G$-symmetric spectra.
\end{Prop}

A $G\times\Sigma_n$-map is a $G\times\Sigma_n$-cofibration
if and only if its underlying map is a cofibration.
Therefore a map of $G$-symmetric spectra is a $G$-flat cofibration if and only if its underlying
morphism of symmetric spectra is a flat cofibration
and hence it suffices to prove the above proposition for $G$ the trivial group.

\begin{subsubsection}{Latching objects of symmetric spectra}
Let $\bf{k}$ denote the set $\{1, \ldots, k\}$.
Those are the objects of the category $\mathcal{I}$, where morphisms are injective maps of sets.
The category $\mathcal{I}$ has a symmetric monoidal structure $\sqcup$ given by concatenation
$\bf{m}\sqcup \bf{n} = \bf{m+n}$ with unit the empty set $\bf{0}$.
Let $(\sqcup \downarrow \bf{n})$ be the category with objects consisting of tuples $(\bf{k}, \bf{k'}, \alpha\colon \bf{k}\sqcup \bf{k'}\rightarrow \bf{n})$
with $\alpha$ injective. A morphism
$(\bf{k}, \bf{k'}, \alpha\colon \bf{k}\sqcup \bf{k'}\rightarrow \bf{n})\rightarrow (\bf{l}, \bf{l'}, \beta\colon \bf{l}\sqcup \bf{l'}\rightarrow \bf{n})$
is a tuple of morphisms $(\gamma\colon \bf{k}\rightarrow \bf{l}, \gamma'\colon \bf{k'}\rightarrow\bf{l'})$ in $\mathcal{I}$
such that $\beta\circ (\gamma\sqcup \gamma') = \alpha$.

Given two symmetric spectra $E$ and $F$, their smash product is given in level $n$ by
\[
 \xymatrix{
 \colim_{\alpha\colon {\bf{k}}\sqcup {\bf{k'}}\rightarrow {\bf{n}}} E({\bf{k}})\wedge F({\bf{k'}})\wedge S^{{{\bf{n}}} - \alpha},
 }
\]
where the colimit is taken over the category $(\sqcup \downarrow {\bf{n}})$ and we use $\alpha$ as a shorthand for the image of $\alpha$.
A map $(\gamma, \gamma')$ in this category induces the map
\[
 E({\bf{k}})\wedge F({\bf{k'}}) \wedge S^{{\bf{n}}-\alpha}\cong E({\bf{k}})\wedge S^{{\bf{l}}-\gamma} \wedge F({\bf{k'}})\wedge S^{\bf{l'}-\gamma'}\wedge S^{\bf{n}-\beta}\rightarrow E({\bf{l}})\wedge F({\bf{l'}})\wedge S^{\bf{n}-\beta},
\]
where one uses $\gamma$ and $\gamma'$ to identify $\bf{n}-\alpha$ with $({\bf{n}} - \beta) \sqcup ({\bf{l}}-\gamma) \sqcup ({\bf{l'}}-\gamma')$
in the first isomorphism and the second map uses the isomorphisms $E({\bf{k}})\cong E(\gamma)$ $F({\bf{k'}})\cong F(\gamma')$
given by $\gamma$ and $\gamma'$ and the generalized structure maps.

Define $\overline{\Sp}$ to be the truncated sphere spectrum, i.e. $\overline{\Sp}_0 = *$
and $\overline{\Sp}_n = S^n$ if $n\geq 1$. The structure maps are the evident maps.
The $n$th latching object of a symmetric spectrum $X$ is now defined to be the $n$th level of the smash
product of $X$ with $\overline{\Sp}$,
\[
L_n(X) = (X\wedge \overline{\Sp})_n.
\]
More generally for a morphism $f\colon X\rightarrow Y$ of symmetric spectra
we set
\[
L_n(f) = X({\bf{n}})\cup_{L_n(X)} L_n(Y).
\]
The generalized structure maps induce
$\nu_n(X)\colon L_n(X)\rightarrow X({\bf{n}})$ and
$\nu_n(f)\colon L_n(f)\rightarrow Y({\bf{n}})$, which are the maps that appear in the definition of the ($G$-)flat model structure.

For our purpose, it is convenient to use a slightly different model for the latching morphisms.
To this end, we define
$\mathcal{P}({\bf{n}})$ to be the poset of subsets
of ${\bf{n}}$.
Given a symmetric spectrum $X$ and a morphism $f\colon X\rightarrow Y$ we
get two functors ${\bf{L}}_n(X)$ and ${\bf{L}}_n(f)$
from $\mathcal{P}({\bf{n}})$ to $\sset$.
On objects, these are given by
$U\mapsto X(U)\wedge S^{{\bf{n}}-U}$ and
$U\mapsto X({\bf{n}})\cup_{X(U)\wedge S^{{\bf{n}}-U}}Y(U)\wedge S^{{\bf{n}}-U}$
respectively.
For an inclusion $\iota\colon U\subset V$ we let ${\bf{L}}_n(X)(\iota)$ be the composite
\[
 \xymatrix{
 X(V)\wedge S^{{\bf{n}}-V} \cong X(V)\wedge S^{U-V}\wedge S^{{\bf{n}}-U}\ar[r] & X(U)\wedge S^{{\bf{n}}- U}
 }
\]
where the second map is given by the generalized structure map $\sigma_V^{U-V}$ smashed with the identity on $S^{{\bf{n}}-U}$ and
similarly for ${\bf{L}}_n(f)$.
There are canonical maps $\tilde{\nu}_n(X)\colon\colim_{U\subsetneq {\bf{n}}}{\bf{L}}_n(X)\rightarrow X({\bf{n}})$
and $\tilde{\nu}_n(f)\colon\colim_{U\subsetneq {\bf{n}}}{\bf{L}}_n(f)\rightarrow Y({\bf{n}})$
induced by the generlized structure maps and we have
\begin{Lemma}
\label{lem: latching objects colimit}
The spaces $L_n(X)$ and $\colim_{U\subsetneq {\bf{n}}}{\bf{L}}_n(X)$ are naturally isomorphic
as $\Sigma_n$-spaces over $X({\bf{n}})$.
Similarly, the spaces $L_n(f)$ and $\colim_{U\subsetneq {\bf{n}}}{\bf{L}}_n(f)$ are naturally isomorphic as $\Sigma_n$-spaces over $Y({\bf{n}})$.
\end{Lemma}
\begin{proof}
Indeed, we define a $\Sigma_n$-map
\[
 L_n(X)\rightarrow \colim_{U\subsetneq {\bf{n}}}{\bf{L}}_n(X)
\]
by mapping $(\alpha\colon {\bf{k}}\sqcup{\bf{k'}}\rightarrow{\bf{n}}, x\wedge y\wedge z\in X({\bf{k}})\wedge S^{{\bf{k'}}}\wedge S^{{\bf{n}}-\alpha})$
to $(\alpha({\bf{k}}), [x, \alpha|_{\bf{k}}]\wedge ((\alpha|_{\bf{k'}})_*(y)\wedge z)\in X(\alpha({\bf{k}}))\wedge S^{{\bf{n}}- \alpha|_{\bf{k}}}$.
By abuse of notation, we secretely identified $X({\bf{k}})$ with $X_k$ via the isomorphism $[x\wedge f]\mapsto f_*(x)$.
The inverse is then given by $(U, [x, \alpha]\wedge y)\mapsto (\alpha\colon {\bf{k}}\rightarrow U\subset {\bf{n}}, x\wedge y)$.
The second part follows since colimits commute with each other.
\end{proof}
\end{subsubsection}

\begin{subsubsection}{A characterization of flat cofibrations}
In order to give a characterization of flat cofibrations, we
need the following lemma.
\begin{Lemma}
\label{lem: characterization flatness}
Given a functor $C\colon \mathcal{P}({\bf{n}})\rightarrow \sset$,
the induced map $\colim_{V\subsetneq U} C(V)\rightarrow C(U)$
is a cofibration for all $U\subset{\bf{n}}$ if and only if
\begin{itemize}
 \item[(a)] for all inclusions $V\subset U\subset{\bf{n}}$, the map $C(V)\rightarrow
 C(U)$ is a cofibration and
 \item[(b)] for all $U$, $V\subset{\bf{n}}$, the intersection of the images of $C(U)$
 and $C(V)$ in $C(U\cup V)$ equals the image of $C(U\cap V)$.
\end{itemize}
\end{Lemma}
\begin{proof}
 This appears in the proof of~\cite[Proposition 3.11]{Sagave Schlichtkrull}.
\end{proof}
We can now prove
\begin{Prop}
\label{prop: characterization flatness}
 A map $f\colon X\rightarrow Y$ of symmetric spectra is a flat cofibration if and only if
 \begin{itemize}
 \item[(a)] for all $k$, $l\geq 0$ the map $X({\bf{k}\sqcup\bf{l}})\cup_{X({\bf{k}})\wedge S^{\bf{l}}} Y({\bf{k}})\wedge S^{{\bf{l}}}\rightarrow Y({\bf{k}}\sqcup{\bf{l}})$ is a cofibration and
 \item[(b)] for all integers $k$, $l$, $m\geq 0$, we have that
 \small
 \[
 \xymatrix{
X({\bf{k}\sqcup \bf{l}\sqcup\bf{m}})\cup_{X({\bf{l}})\wedge S^{\bf{k}\sqcup \bf{m}}} Y({\bf{l}})\wedge S^{\bf{k}\sqcup\bf{m}} \ar[r]\ar[d] & X({\bf{k}\sqcup \bf{l}\sqcup\bf{m}})\cup_{X({\bf{l}}\sqcup {\bf{m}})\wedge S^{\bf{k}}} Y({\bf{l}}\sqcup {\bf{m}})\wedge S^{\bf{k}}\ar[d]\\
X({\bf{k}\sqcup \bf{l}\sqcup\bf{m}})\cup_{X({\bf{k}}\sqcup {\bf{l}})\wedge S^{\bf{m}}} Y({\bf{k}}\sqcup {\bf{l}})\wedge S^{\bf{m}}\ar[r] & Y({\bf{k}\sqcup \bf{l}\sqcup\bf{m}})
}
\]
\end{itemize}
\normalsize
is a pullback.

In particular, if $Y$ is a flat symmetric spectrum and $X\subset Y$ is a subspectrum, then the inclusion $X\rightarrow Y$ is a flat cofibration if and only
if for all $k$, $l\geq 0$ the intersection of the images of $X({\bf{k}\sqcup\bf{l}})$ and $Y({\bf{k}})\wedge S^{\bf{l}}$ in $Y({\bf{k}\sqcup\bf{l}})$
equals the image of $X({\bf{k}})\wedge S^{\bf{l}}$.
\end{Prop}
\begin{proof}
In view of Lemma \ref{lem: latching objects colimit},
a map of symmetric spectra $f\colon X\rightarrow Y$ is flat
if and only if for all $n\geq 0$ and all subsets $U\subset {\bf{n}}$
the maps $\colim_{V\subsetneq U} X(U)\cup_{X(V)\wedge S^{U-V}} Y(V)\wedge S^{U-V}\rightarrow Y(U)$
are cofibrations. By Lemma \ref{lem: characterization flatness}, this is equivalent to conditions $(a)$ and $(b)$.
\end{proof}

We can now give a proof of the result we are after.
\begin{proof}[Proof of Proposition \ref{prop: product G flat}]
 Suppose $X$ is a flat symmetric spectrum.
 We prove first that $X^{\times N}$ is flat, provided that $X$ is flat.
 Condition $(a)$ in Proposition \ref{prop: characterization flatness}
 requires the map $X({\bf{n}})^{\times N}\wedge S^{\bf{k}}\rightarrow X({\bf{n}}\sqcup {\bf{k}})^{\times N}$
 to be a cofibration. But this map factors as the composition of two cofibrations
 \[
  \xymatrix{
X({\bf{n}})^{\times N}\wedge S^{\bf{k}}\rightarrow (X({\bf{n}})\wedge S^{\bf{k}})^{\times N} \rightarrow X({\bf{n}}\sqcup {\bf{k}})^{\times N}.
  }
 \]
Condition $(b)$ requires
\[
 \xymatrix{
 X({\bf{l}})^{\times N}\wedge S^{{\bf{k}}\sqcup {\bf{m}}} \ar[r] \ar[d] & X({\bf{l}}\sqcup{\bf{m}})^{\times N}\wedge S^{{\bf{k}}}\ar[d]\\
 X({\bf{k}}\sqcup{\bf{l}})^{\times N}\wedge S^{\bf{m}}\ar[r] & X({\bf{k}}\sqcup{\bf{l}}\sqcup{\bf{m}})^{\times N}
 }
\]
to be a pullback, which is readily checked.
It follows now from the second part of Proposition \ref{prop: characterization flatness} that $X^{\times N}_{\leq N-1}\rightarrow X^{\times N}$ is a flat cofibration.
\end{proof}
\end{subsubsection}
\end{subsection}

\end{section}

\end{document}